\documentclass[11pt]{article}
\usepackage[latin1]{inputenc}
\usepackage{epsfig}
\usepackage{color}
\usepackage[british,english]{babel}
\usepackage{amsthm}
\usepackage{amsmath}
\usepackage{amsfonts}
\usepackage{amssymb}
\usepackage{graphicx}
\numberwithin{equation}{section}

\newtheorem{corollary}{Corollary}[section]

\newtheorem{lemma}[corollary]{Lemma}
\newtheorem{proposition}[corollary]{Proposition}
\newtheorem{remark}[corollary]{Remark}
\newtheorem{theorem}[corollary]{Theorem}
\newfont{\sBlackboard}{msbm10 scaled 900}

\newcommand{\mylabel}[1]{\label{#1}
            \ifx\undefined\stillediting
            \else \fbox{$#1$}\fi }
\newcommand{\BE}{\begin{equation}}

\newcommand{\EEQ}{\end{equation}}
\newcommand{\rfb}[1]{\mbox{\rm
   (\ref{#1})}\ifx\undefined\stillediting\else:\fbox{$#1$}\fi}

\newfont{\Blackboard}{msbm10 scaled 1200}

\newfont{\roma}{cmr10 scaled 1200}

\def\CC{\rm \hbox{C\kern-.56em\raise.4ex
         \hbox{$\scriptscriptstyle |$}\kern+0.5 em }}

\def\Frac{\displaystyle\frac}
\def\Int{\displaystyle\int}
\def\n{|\kern -.05cm{|}\kern -.05cm{|}}

\def\R{{\bf \hbox{\sc I\hskip -2pt R}}} 

\newcommand{\mm}    {{\hbox{\hskip 0.5pt}}}

\newcommand{\bluff} {{\hbox{\raise 15pt \hbox{\mm}}}}
\makeatletter
\def\section{\@startsection {section}{1}{\z@}{-3.5ex plus -1ex minus
    -.2ex}{2.3ex plus .2ex}{\large\bf}}
\makeatother
\def\be{\begin{equation}}
\def\ee{\end{equation}}

\date{ }
\begin{document}
\thispagestyle{empty}
\title{\bf Existence, uniqueness and global behavior of the solutions to some  nonlinear vector equations in  a finite dimensional Hilbert space }\maketitle

\author{ \center  Mama ABDELLI\\
Laboratory of Analysis and Control of Partial Differential Equations,\\
Djillali Liabes University, P. O. Box 89, Sidi Bel Abbes 22000,
Algeria.\\
abdelli$_{-}$mama@yahoo.fr\\}
\medskip\author{ \center  Mar\'ia ANGUIANO \\ Departamento de An\'alisis Matem\'atico, Facultad de Matem\'aticas,\\ Universidad de Sevilla, P. O. Box 1160, 41080-Sevilla, Spain\\ anguiano@us.es\\}
\medskip\author{ \center  Alain HARAUX \\ 
Sorbonne Universit\'es, UPMC Univ Paris 06, CNRS, UMR 7598, \\Laboratoire Jacques-Louis Lions, 4, place Jussieu 75005, Paris, France.\\
 haraux@ann.jussieu.fr\\}

\vskip20pt

 \renewcommand{\abstractname} {\bf Abstract}
\begin{abstract} The initial value problem and global properties of solutions are studied for  the vector  equation: $\Big(\|u'\|^{l}u'\Big)'+\|A^{\frac{1}{2}}u\|^\beta Au+g(u')=0$ in a finite dimensional Hilbert space under suitable assumptions on $g$.
\end{abstract}
\bigskip\noindent

 {\small \bf AMS classification numbers:} 34A34, 34C10, 34D05, 34E99

\bigskip\noindent {\small \bf Keywords:}  Existence of the solution, Uniqueness of the solution, Decay rate.\newpage
\section {Introduction}
Let $H$ be a finite dimensional  real  Hilbert space, with norm denoted by $\|.\|$. We consider the following nonlinear equation
 \begin{equation}\label{1}
 \Big(\|u'\|^{l}u'\Big)'+\|A^{\frac{1}{2}}u\|^\beta Au+g(u')=0,
 \end{equation}
 where  ${ l}$ and $\beta$ are positive constants,
 and $A$ is a positive and symmetric  linear operator on $H$. We denote by $\left(\cdot,\cdot \right)$ the inner product  in $H$. The operator $A$ is coercive, which means :
 $$\exists \lambda > 0,\,\,\, \forall u\in D(A),\,\,\,\, (Au,u)\geq \lambda \|u\|^2.
 $$
We also define
 $$
 \forall u\in H,\,\, \|A^{\frac{1}{2}}u\|: =\|u\|_{ D(A^{\frac{1}{2}})},
 $$
 a norm equivalent to  the norm in $H$. We assume that  $g:H\rightarrow H $ is locally Lipschitz continuous.

 When ${ l}=0$ and $g(u')=c|u'|^\alpha u'$ Haraux  {\bf\cite{har2}} studied the rate of decay  of the energy of non-trivial solutions to the scalar second order ODE. In addition, he showed that if $\alpha > \frac{\beta}{\beta+2}$ all non-trivial solutions are oscillatory and if $\alpha < \frac{\beta}{\beta+2}$ they are non-oscillatory. In the oscillatory case he established that all non-trivial solutions have the same decay rates, while in the non-oscillatory case he showed the coexistence of exactly two different decay rates, calling slow solutions those which have the lowest decay rate, and fast solutions the others. 

 Abdelli and  Haraux {\bf\cite{AMH}} studied the scalar second order ODE where $g(u')=c|u'|^\alpha u'$, they proved the existence and uniqueness of a global solution with initial data $(u_0,u_1)\in \mathbb{R}^2$. They used some modified energy function to estimate the rate of decay and they used the method introduced by Haraux {\bf\cite{har2}} to study the oscillatory or non-oscillatory of non-trivial solutions. If $\alpha > \frac{\beta(l+1)+l}{\beta+2}$ all non-trivial solutions are oscillatory and if $\alpha < \frac{\beta(l+1)+l}{\beta+2}$ they are non-oscillatory. In the non-oscillatory rate, as in the case ${ l}=0$, the coexistence of exactly two different decay rates was established.

 In this article, we use some techniques from  Abdelli and  Haraux {\bf\cite{AMH}} to establish a global existence and uniqueness result of the solutions, and under some additional conditions on $g$ (typically $g(s)\sim c\|s\|^\alpha s$), we study the asymptotic behavior as $t\rightarrow \infty$. A basic role will be played by the total energy of the solution $u$ given by the formula 
\begin{eqnarray}\label{energy} E(t) = \frac{l+1}{l+2}\|u'(t)\|^{l+2}+\frac{1}{\beta+2}\|A^{\frac{1}{2}}u(t)\|^{\beta+2}.\end{eqnarray}

 The plan of this paper is as follows: In Section $2$ we establish some basic preliminary inequalities, and in Section $3$ we prove the existence of a solution  $u\in {\mathcal C}^1(\mathbb{R}^+,H)$ with $\|u'\|^{l}u'\in {\mathcal C}^1(\mathbb{R}^+,H)$ for any initial data $(u_0,u_1)\in H \times H$ under relevant conditions on $g$ and the conservation of total energy for each such solution. In Section $4$ we establish the uniqueness of the solution in the same regularity class under additional conditions on $g$. In Section $5$ we prove convergence of all solutions to $0$ under more specific conditions on $g$ and we estimate the decay rate of the energy. Finally, in Section 6, we discuss the optimality of these estimates when $g(s) = c\|s\|^\alpha s$ and $l< \alpha< \frac{\beta(1+l)+l}{\beta+2}$;  in particular, by relying on a technique introduced by Ghisi, Gobbino and Haraux {\bf\cite{GGH}}, we prove the existence of a open set of initial states giving rise to slow decaying solutions. In our last result, by relying on a technique introduced by Ghisi, Gobbino and Haraux \cite{GhGoHa}, we prove that all non-zero solutions are either slow solutions or fast solutions. 

 \section{Some basic inequalities}
 In this section, we establish some easy but powerful lemmas which generalize  Lemma 2.2 and  Lemma 2.3 from {\bf\cite{AIH}} and will be essential for the existence and uniqueness proofs of the next section. Troughout this section, $H$ denotes an arbitrary (not necessary finite dimensional) real Hilbert space with norm denoted by $\|.\|$.
  \begin{proposition}
Let $ (u, v) \in H\times H $ and $(\alpha, \beta)\in \R^+\times \R^+$. Then the condition 
$$ (\alpha- \beta)(\|u\|- \|v\|)\ge 0, $$ implies $$ (\alpha u- \beta v)(u-v) \ge \frac{1}{2}(\alpha+\beta)\|u-v\|^2.$$
\end{proposition}
\begin{proof} An easy calculation gives the identity
$$ (\alpha u- \beta v)(u-v) =  \frac{1}{2}(\alpha- \beta)(\|u\|^2- \|v\|^2)+ \frac{1}{2}(\alpha+\beta)\|u-v\|^2.$$
Since $$ (\alpha- \beta)(\|u\|^2- \|v\|^2)= (\|u\|+ \|v\|)(\alpha- \beta)(\|u\|- \|v\|),$$ 
the result follows immediately.
\end{proof}

For the next results, we consider a number $R>0$ and we set 
$$J_R = [0, R]; \quad  B_R : = \{ u\in H, \,\,\, \|u\|\le R\}$$

\begin{proposition}
Let $f: J_R \longrightarrow {\R}^+$ be non increasing and such that for some positive numbers $p$, $c$:
$$ \forall s\in J_R ,\quad  f(s) \ge c s^p.  $$

 Then we have  
$$ \forall (u, v) \in B_R\times B_R,\quad  (f(\|u\|) u- f(\|v\|) v)(u-v) \ge \frac{c}{2}(\|u\|^p+\|v\|^p)\|u-v\|^2. $$
\end{proposition}
\begin{proof} Applying the previous result with $\alpha = f(\|u\|)$ and $\beta = f(\|v\|)$ the result follows immediately. \end{proof}
\begin{corollary}\label{coerc}
Let $f: J_R \longrightarrow {\R}^+$ be non increasing and such that for some positive numbers  $p$, $c$:
$$ \forall s\in J_R ,\quad  f(s) \ge c s^p.  $$

 Then we have 
 $$ \forall (u, v) \in B_R\times B_R,\quad  \|f(\|u\|) u- f(\|v\|) v)\| \ge \frac{c}{2}(\|u\|^p+\|v\|^p)\|u-v\|. $$
  \end{corollary}
\begin{proof} This inequality follows immediately from Cauchy-Schwarz inequality combined with the conclusion of the previous proposition. \end{proof}
\begin{proposition}
Let $f: J_R \longrightarrow {\R}^+$ be non increasing and such that for some positive numbers  $p$, $c$:
$$ \forall s\in J_R ,\quad  f(s) \ge c s^p.  $$

 Then we have  
$$ \forall (u, v) \in B_R\times B_R,\quad  (f(\|u\|) u- f(\|v\|) v)(u-v) \ge \delta\|u-v\|^{p+2}, $$ with $ \delta = \frac{c}{2^{\max\{p,1\}}}$.  
\end{proposition}
\begin{proof} It is sufficient to apply the previous result combined with the inequality
 $$\|u-v\|^{p} \le 2^{\max\{p-1,0\}}(\|u\|^p+\|v\|^p).$$ \end{proof}
 
 \begin{corollary}\label{hold}
Let $f: J_R \longrightarrow {\R}^+$ be non increasing and such that for some positive numbers  $p$, $c$:
$$ \forall s\in J_R ,\quad  f(s) \ge c s^p.  $$

 Then we have for some constant $C= C(c, p)>0, $ 
$$ \forall (u, v) \in B_R\times B_R,\quad  \|u-v\|\le C \|f(\|u\|) u- f(\|v\|) v)\|^{\frac{1}{p+1}}. $$ \end{corollary}
\begin{proof} This inequality follows immediately from Cauchy-Schwarz inequality combined with the conclusion of the previous proposition. \end{proof}
\begin{lemma}\label{LEM12} Assume that $A$ is a positive, symmetric, bounded operator on $H$. for some constant $D>0$ we have 
$
\forall (w,v)\in H\times H$
$$
\forall (w,v)\in H\times H,\,\,\|\|A^{\frac{1}{2}}w\|^{\beta}Aw-\|A^{\frac{1}{2}}v\|^{\beta}Av\|\leq D\max(\|A^{\frac{1}{2}}v\|,\|A^{\frac{1}{2}}w\|)^{\beta}\|w-v\|.
$$
\end{lemma}
\begin{proof}
 We can write
{\small\begin{equation*}
 \begin{split}
\|\|A^{\frac{1}{2}}w\|^{\beta}Aw-\|A^{\frac{1}{2}}v\|^{\beta}Av\|&= \|A^{\frac{1}{2}}(\|A^{\frac{1}{2}}w\|^{\beta}A^{\frac{1}{2}}w-\|A^{\frac{1}{2}}v\|^{\beta}A^{\frac{1}{2}}v)\|
\\&\leq\|A^{\frac{1}{2}}\|\|\|A^{\frac{1}{2}}w\|^{\beta}A^{\frac{1}{2}}w-\|A^{\frac{1}{2}}v\|^{\beta}A^{\frac{1}{2}}v\|.
\end{split}
\end{equation*}}
Direct calculations (see also {\bf\cite{AIH}}, Lemma 2.2)  yield
{\small\begin{equation*}
 \begin{split}
\|\|A^{\frac{1}{2}}w\|^{\beta}A^{\frac{1}{2}}w-\|A^{\frac{1}{2}}v\|^{\beta}A^{\frac{1}{2}}v\|&\leq C_1\max(\|A^{\frac{1}{2}}v\|,\|A^{\frac{1}{2}}w\|)^{\beta}\|A^{\frac{1}{2}}w-A^{\frac{1}{2}}v\|\\&\leq C_2\max(\|A^{\frac{1}{2}}v\|,\|A^{\frac{1}{2}}w\|)^{\beta}\|w-v\|,
\end{split}
\end{equation*}}
and the result is proved.
\end{proof}

 \section{Global existence and energy conservation for equation (\ref{1})}
 In this section, we study the existence of a solution for the initial value problem associated to equation (\ref{1}) where $g:H\rightarrow H$
is a locally Lipschitz continuous  function which satisfies the following hypothesis:
 \begin{equation}\label{S1}
 \exists k_1> 0,\,\,\,\, k_2> 0,\,\,\, \forall v,\,\,\,  (g(v),v)\geq -k_1-k_2\|v\|^{l+2}.
 \end{equation}
\begin{theorem}\label{Pr1}
Let $(u_0,u_1)\in H\times H$.
The problem (\ref{1}) has a  global solution satisfying
$$u\in {\mathcal C}^1(\mathbb{R}^+,H),\,\,\,\,\,\,\|u'\|^{l}u'\in {\mathcal C}^1(\mathbb{R}^+,H)\,\,\,\,\
\mbox{and}\,\,\,\, u(0)= u_0,\,\,\ u'(0)= u_1.$$
\end{theorem}
\begin{proof}
To show the existence of the solution for (\ref{1}), we consider the auxiliary problem
{\small\begin{equation}\label{H1}\left\{
\begin{array}{ll}
(\varepsilon + \|u_\varepsilon '\|^2)^{l/2} u_\varepsilon'' +l(u_\varepsilon',u_\varepsilon'')(\varepsilon + \|u_\varepsilon'\|^2)^{l/2 -1} u_\varepsilon'
+\|A^{\frac{1}{2}}u_\varepsilon\|^\beta Au_\varepsilon+g(u_\varepsilon')=0,\\
u_\varepsilon(0)=u_0,\,\,\,\, u'_\varepsilon(0)=u_1.
\end{array} \right.\end{equation}} Here, $\varepsilon > 0$ is a small parameter, devoted to tend to zero. For simplicity in the sequel we shall write $$ l/2: = m>0.$$
Assuming the existence of such a solution $u_\varepsilon$,
multiplying (\ref{H1}) by $u'_\varepsilon$, we find
$$
[(\varepsilon + \|u_\varepsilon'\|^2)^{m}+2m (\varepsilon + \|u_\varepsilon'\|^2)^{m-1} \|u_\varepsilon'\|^2] (u_\varepsilon',u_\varepsilon'')
+\|A^{\frac{1}{2}}u_\varepsilon\|^\beta (Au_\varepsilon,u_\varepsilon')+(g(u_\varepsilon'),u_\varepsilon')=0,
$$
then
\begin{equation}\label{XE1}
 (u_\varepsilon',u_\varepsilon'')=
-\Frac{\|A^{\frac{1}{2}}u_\varepsilon\|^\beta (Au_\varepsilon,u_\varepsilon')+(g(u_\varepsilon'),u_\varepsilon')}
{(\varepsilon + \|u_\varepsilon'\|^2)^{m-1}(\varepsilon + (l+1)\|u_\varepsilon'\|^2) }.
\end{equation}
From (\ref{H1}), we obtain that $u_\varepsilon$ is a solution of
{\small\begin{equation}\label{HH1}\left\{
\begin{array}{ll}
(\varepsilon + \|u_\varepsilon'\|^2)^{m} u_\varepsilon''-l\Frac{\|A^{\frac{1}{2}}u_\varepsilon\|^\beta (Au_\varepsilon,u_\varepsilon')+(g(u_\varepsilon'),u_\varepsilon')}
{\varepsilon +(l+1)\|u_\varepsilon'\|^{2}}u_\varepsilon'
+\|A^{\frac{1}{2}}u_\varepsilon\|^\beta Au_\varepsilon+g(u_\varepsilon')=0,\\
u_\varepsilon(0)=u_0,\,\,\,\, u'_\varepsilon(0)=u_1.
\end{array} \right.\end{equation}}
Conversely,  (\ref{HH1}) implies (\ref{XE1}).
Then, replacing  (\ref{XE1}) in (\ref{HH1}), we obtain (\ref{H1}). Therefore (\ref{H1}) is equivalent to (\ref{HH1}).
\begin{itemize}
 \item[i)] A priori estimates:\\\\
 Next, we introduce
 \begin{eqnarray*}
F_\varepsilon(u_\varepsilon,u_\varepsilon')&=&l\Frac{\|A^{\frac{1}{2}}u_\varepsilon\|^\beta (Au_\varepsilon,u_\varepsilon')+(g(u_\varepsilon'),u_\varepsilon')}
{{ (\varepsilon + \|u_\varepsilon'\|^2)^{m}}(\varepsilon +(l+1)\|u_\varepsilon'\|^{2})}u_\varepsilon'
-\Frac{\|A^{\frac{1}{2}}u_\varepsilon\|^\beta Au_\varepsilon+g(u_\varepsilon')}{{ (\varepsilon + \|u_\varepsilon'\|^2)^{m}}}.
\end{eqnarray*}
Then (\ref{HH1}), can be rewritten as
{\small\begin{equation}\label{HWW2}\left\{
\begin{array}{ll}
u_\varepsilon''
=F_\varepsilon(u_\varepsilon,u_\varepsilon'),\\
u_\varepsilon(0) =u_0,\,\,\,\, u'_\varepsilon(0) =u_1.
\end{array} \right.\end{equation}}
Since the vector field $F_\varepsilon$ is locally Lipschitz continuous, the existence and uniqueness of $u_\varepsilon$ in the class ${\mathcal C}^2([0,T),H)$, for some $T> 0$ is classical. Multiplying (\ref{H1}) by $u'_\varepsilon$, we obtain by a simple calculation the following energy identity:
{\small\begin{equation*}
 \begin{split}
 \frac{d}{dt}E_\varepsilon(t)
 +(g(u_\varepsilon'(t)),u_\varepsilon'(t)) = 0,
 \end{split}
\end{equation*}}
where
 $$E_\varepsilon(t)
 =\frac{l+1}{l+2}(\varepsilon + \|u_\varepsilon'\|^2)^{m+1}- \varepsilon(\varepsilon + \|u_\varepsilon'\|^2)^{m}+\frac{1}{\beta+2}\|A^{\frac{1}{2}}u_\varepsilon(t)\|^{\beta+2}.
 $$ Indeed, for any function $ v\in \mathcal C ^1([0,T),H)$ we have the following sequence of identities
 $$ ( ((\varepsilon + \| v \|^2)^m v)', v) = 2m(\varepsilon + \| v \|^2)^{m-1} \| v \|^2 (v, v') +(\varepsilon 
 + \| v \|^2)^m  (v, v') $$
 $$ =  2m(\varepsilon + \| v \|^2)^{m-1} (\varepsilon+\| v \|^2) (v, v')  - 2m\varepsilon (\varepsilon + \| v \|^2)^{m-1} (v, v') + (\varepsilon 
 + \| v \|^2)^m  (v, v') $$
 $$ =  (2m+1)(\varepsilon 
 + \| v \|^2)^m  (v, v')  - 2m\varepsilon (\varepsilon + \| v \|^2)^{m-1} (v, v') $$
 $$ =\frac{d}{dt}\left[\frac{l+1}{l+2}(\varepsilon + \|{v}\|^2)^{m+1}- \varepsilon(\varepsilon + \|{v}\|^2)^{m}\right]. $$ Moreover for some constant $C>0$ independent of $ \varepsilon$, we have
 $$E_\varepsilon(t) +C
 \ge \frac{l+1}{l+2}(\varepsilon + \|u_\varepsilon'\|^2)^{m+1}- \varepsilon(\varepsilon + \|u_\varepsilon'\|^2)^{m}+C \ge \frac{1}{4}\|u_\varepsilon'(t)\|^{l+2}$$
and then as  a consequence of (\ref{S1}) 
{\small\begin{equation*}
 \begin{split}
\frac{d}{dt}{E_\varepsilon(t)}=-(g(u_\varepsilon'(t)),u_\varepsilon'(t))&\leq k_1+k_2\|u_\varepsilon'(t)\|^{l+2}\\&\leq k_3+k_4{E_\varepsilon(t)}.
 \end{split}
\end{equation*}}By Gronwall's inequality, this  implies
\begin{equation}\label{H4}
\forall t \in [0,T), \quad \|u_\varepsilon(t)\|\leq M_1,\,\,\,\|u'_\varepsilon(t)\|\leq M_2,
\end{equation}
 for some constants $M_1, M_2$ independent of $\varepsilon$. Hence,
$u_\varepsilon$ and $u'_\varepsilon$ are uniformly bounded and $u_\varepsilon$ is a global solution, in particular $T>0$  can be taken arbitrarily large.
\item[ii) ]
 Passage to the limit:\\\\
 In order to pass to the limit as $\varepsilon \rightarrow 0$ we need to know that $u_\varepsilon$ and  $u'_\varepsilon$ are uniformly equicontinous on $(0,T)$ for any $T> 0$. For $u_\varepsilon$ it is clear, since $\|u'_\varepsilon\|$ is bounded.\\
 Moreover we  have $$(\varepsilon + \|u_\varepsilon '\|^2)^{l/2} u_\varepsilon'' +l(u_\varepsilon',u_\varepsilon'')(\varepsilon + \|u_\varepsilon'\|^2)^{l/2 -1} u_\varepsilon'=  ((\varepsilon + \|u_\varepsilon '\|^2)^{l/2} u_\varepsilon')',$$ hence by the equation $$ \|((\varepsilon + \|u_\varepsilon '\|^2)^{l/2} u_\varepsilon')'\|\leq \|A^{\frac{1}{2}}u_\varepsilon(t)\|^{\beta}\|Au_\varepsilon(t)\|+\|g(u_\varepsilon'(t))\|,
$$
and since $g$ is locally Lipschitz continuous, hence bounded on  bounded sets, we obtain 
$$ \|((\varepsilon + \|u_\varepsilon '\|^2)^{l/2} u_\varepsilon')'\|\leq M_3.$$
 Therefore the functions $(\varepsilon + \|u_\varepsilon '\|^2)^{m} u_\varepsilon'$ are uniformly Lipschitz continuous on $\mathbb{R}+$.  We claim that  $u'_\varepsilon$ is  a uniformly (with respect to $\varepsilon$) locally  H\"older continuous function of $t$.  Indeed by applying Corollary \ref{hold} with $f(s) = (\varepsilon+ s^2)^{m}$  and $p= 2m = l$ we find for some $C>0$
 $$\|u_{\varepsilon}'(t_1)-u_{\varepsilon}'(t_2)\|\leq C\|(\varepsilon + \|u_\varepsilon '\|^2)^{m} u_\varepsilon' (t_1)-(\varepsilon + \|u_\varepsilon '\|^2)^{m} u_\varepsilon'(t_2)\|^{\frac{1}{l+1}}\leq C'\|t_1-t_2\|^{\frac{1}{l+1}}.$$
As a consequence of Ascoli's theorem and a priori estimate
(\ref{H4}) combined with (\ref{HWW2}), we may extract a subsequence which is still denoted for simplicity by
$(u_\varepsilon)$ such that for every $T> 0$
$$
u_\varepsilon \rightarrow u\,\,\,\,\,\mbox{in}\,\,\, {\mathcal C}^1((0,T),H),
$$
as $\varepsilon$ tends to $0$. Integrating (\ref{H1}) over $(0,t)$, we get
 \begin{eqnarray}\label{li11}
  (\varepsilon + \|u_\varepsilon '\|^2)^{m} u_\varepsilon' (t)&-&(\varepsilon + \|u_\varepsilon '\|^2)^{m} u_\varepsilon' (t)(0)\\ &=&
 -\Int_0^t\|A^{\frac{1}{2}}u_\varepsilon(s)\|^\beta A u_\varepsilon(s)\,ds-
 \Int_0^tg(u'_\varepsilon(s))\,ds.
 \nonumber
 \end{eqnarray}
From (\ref{li11}), we then have, as $\varepsilon$ tends to $0$
$$
(\varepsilon + \|u_\varepsilon '\|^2)^{m} u_\varepsilon' (t) \rightarrow -\Int_0^t\|A^{\frac{1}{2}}u(s)\|^\beta A u(s)\,ds
-\Int_0^tg( u'(s))\,ds
+\|u'(0)\|^{l}u'(0)
\,\,\,\,\,\mbox{in}\,\,\, {\mathcal C}^0((0,T),H).
$$
Hence \begin{equation}\label{ABS}
\|u'\|^{l}u'= -\Int_0^t\|A^{\frac{1}{2}}u(s)\|^\beta A u(s)\,ds-\Int_0^tg( u'(s))\,ds
+\|u'(0)\|^{l}u'(0),
\end{equation}
and $\|u'\|^{l}u' \in {\mathcal C}^1((0,T),H)$. Finally by differentiating (\ref{ABS}) we conclude that $u$ is a solution of (\ref{1}).
\end{itemize}
\end{proof}
\begin{remark}
It is not difficult to see that the solution $u$ constructed in the existence theorem satisfies the energy identity  $\frac{d}{dt}E(t)=-(g(u'(t)),u'(t))$. The following stronger result shows that this identity is true  for any solution even if uniqueness is not known. For infinite dimensional equations such as the Kirchhoff equation, both uniqueness and the energy identity for general weak solutions  are old open problems. 
\end{remark}

\begin{theorem}\label{Energ}
Let $(u_0,u_1)\in H\times H$.
Then any solution $u$ of { (\ref{1}) such that}
$$u\in {\mathcal C}^1(\mathbb{R}^+,H),\,\,\,\,\,\,\|u'\|^{l}u'\in {\mathcal C}^1(\mathbb{R}^+,H)\,\,\,\,\
\mbox{and}\,\,\,\, u(0)= u_0,\,\,\ u'(0)= u_1,$$ satisfies the formula 
\begin{equation}\label{V1}
\frac{d}{dt}E(t)=-(g(u'(t)),u'(t)).
\end{equation}
\end{theorem} The proof of this new result relies on the following simple lemma
\begin{lemma}\label{LAM}
Let  $J$ be any interval, assume that   $v\in {\mathcal C}(J,H)$ and $\|v\|^{l}v\in {\mathcal C}^1(J,H)$ then
$$
\|v(t)\|^{l+2}\in {\mathcal C}^1(J,H),
$$
and
$$
((\|v(t)\|^{l}v(t))',v(t))=\frac{d}{dt}\Big(\frac{l+1}{l+2}\|v(t)\|^{l+2}\Big),\,\,\,\,\,\,\forall t\in J.
$$
\end{lemma}
\begin{proof}
{\bf{Case 1.}} For any $t_0 \in J$, if $v(t_0)\neq 0$ then $v(t)\neq 0$ in the neighborhood of $t_0$ and in this neighborhood we have $v\in {\mathcal C}^1(J,H)$  with
{\small\begin{equation}\label{hh1hm}
\begin{split}
\frac{d}{dt}\left(\|v(t)\|^{l+2}\right)&=\frac{d}{dt}\left(\|v(t)\|^{2}\right)^{{ l/2}+1}\\&=
\left({\frac{l}{2}}+1\right)(\|v(t)\|^{2})^{{ l/2}}\frac{d}{dt}\left(\|v(t)\|^{2}\right)\\&
=2\left({\frac{l}{2}}+1\right)\|v(t)\|^{l}(v(t),v'(t)).
\end{split}
\end{equation}}
On the other hand, we find
{\small\begin{equation}\label{hh1hh}
\begin{split}
((\|v(t)\|^{l}v(t))',v(t))&=(\|v(t)\|^{l}v'(t),v(t))
+(\|v(t)\|^{l})'(v(t),v(t))\\&=
\|v(t)\|^{l}(v'(t),v(t))+l\|v(t)\|^{l}(v'(t),v(t))\\&=(l+1)\|v(t)\|^{l}(v'(t),v(t)),
\end{split}
\end{equation}}
and from (\ref{hh1hm})-(\ref{hh1hh}), we obtain
\begin{equation*}
((\|v(t)\|^{l}v(t))',v(t))=
\frac{d}{dt}\Big[\frac{l+1}{l+2}\|v(t)\|^{l+2}\Big],\,\,\,\,\, \mbox{for}\,\,\,\, t=t_0.
\end{equation*}
{\bf{Case 2.}} If $v(t_0)= 0$ then since $\|v(t)\|^{l}v(t)\in {\mathcal C}^1(J,H)$ , we have clearly 
$$((\|v\|^{l}v)'(t_0),v(t_0))= 0.
$$ Moreover in the neighborhood of $t_0$ we have
$$\|v(t)\|^{l}v(t) =  0 (|t-t_0|), $$ in other terms $$
\|v(t)\|^{l+1} \leq C |t-t_0|,
$$
therefore
$$
\|v(t)\|^{l+2} \leq C^{\frac{l+2}{l+1}}|t-t_0|^{1+\frac{1}{l+1}},
$$
then
$$
 (\|v(t)\|^{l+2})'=0,\,\,\,\, \mbox{for}\,\,\,\, t=t_0,
$$
and finally \begin{equation*}
((\|v(t)\|^{l}v(t))',v(t))=
\frac{d}{dt}\Big[\frac{l+1}{l+2}\|v(t)\|^{l+2}\Big] = 0\,\,\,\,\, \mbox{for}\,\,\,\, t=t_0.
\end{equation*}
\end{proof}
We now give the proof of Theorem {\ref{Energ}}.

\begin{proof} Setting  $v=u'$, from Lemma \ref{LAM}, we deduce
\begin{equation*}
((\|u'(t)\|^{l}u'(t))',u'(t))=
\frac{d}{dt}\Big[\frac{l+1}{l+2}\|u'(t)\|^{l+2}\Big],\,\,\,\,\,\,\forall t  \in J.
\end{equation*}

By multiplying equation (\ref{1}) by $u'$, we obtain easily
$$
\frac{d}{dt}E(t)=-(g(u'(t)),u'(t)).
$$ \end{proof}

\section{Uniqueness of solution for $(u_0,u_1)$ given}
In this section we suppose that
\begin{equation}\label{SS1}
\forall R\in \R^+,  \forall (u,v)\in B_R\times B_R,\,\,\,\,\,\,
\|g(u)-g(v)\|\leq k_3(R)\|u-v\|,
\end{equation}
for some positive constant $k_3(R)$.

\begin{theorem}\label{Uniq}
Let $(u_0,u_1)\in H\times H$, $J$ an interval of  $\mathbb{R}$ and $t_0\in J$.
Then (\ref{1}) has at most one solution 
$$u\in {\mathcal C}^1(J,H),\,\,\,\,\,\,\|u'\|^{l}u'\in {\mathcal C}^1(J,H)\,\,\,\,\
\mbox{and}\,\,\,\, u(t_0)= u_0,\,\,\ u'(t_0)= u_1.$$
\end{theorem}

\begin{remark}
The uniqueness of solutions of (\ref{1})  will be proved under  conditions on the initial data $(u_0,u_1)$. The next proposition concerns the uniqueness result for  $u_1\neq 0$.
\end{remark}
\begin{proposition}
Let $\tau \in \mathbb{R}^+$ and $J=(\tau,T),\,\, T> \tau$. Then there is at most one solution of (\ref{1}) with $u(\tau)=u_0$ and $u'(\tau)=u_1$ for $T-\tau$ small enough such that
$$u\in {\mathcal C}^1(J,H),\,\,\,\,\,\mbox{and}\,\,\,\,\,\,\|u'\|^{l}u'\in {\mathcal C}^1(J,H).$$
\end{proposition}
\begin{proof}
Since $u'(\tau)\neq 0$, the second derivative $u''(\tau)$ exists and $u''(t)$ also exists for $\tau\leq t< \tau+\varepsilon$ with $\varepsilon$ small enough.\\ On $(T,T+\varepsilon)$, (\ref{1}) reduces to
$$
u''=l\|u'\|^{2({ \frac{l}{2}}-1)}u'\Frac{\|A^{\frac{1}{2}}u\|^\beta (Au,u')+(g(u'),u')}
{(l+1)\|u'\|^{l} \|u'\|^{l}}
-\Frac{\|A^{\frac{1}{2}}u\|^\beta Au+g(u')}{\|u'\|^{l}},
$$
the existence and uniqueness of $u$ in the class ${\mathcal C}^2(J,H)$ for this equation is classical.
 \end{proof}
\begin{proposition}\label{Propu}
Let $a \not= 0$.
Then for $J$ an interval containing $0$ and such that $\vert J\vert$ is small enough, equation \eqref{1} has at most one  solution satisfying
$$u\in {\mathcal C}^1(J,H),\,\,\,\,\,\,\|u'\|^{l}u'\in {\mathcal C}^1(J,H)\,\,\,\,\
\mbox{and}\,\,\,\, u_0=a,\,\,\ u_1=0.$$
\end{proposition}
\begin{proof}
 From (\ref{1}), we obtain
 $$
 (\|u'\|^{l}u')'(0)=-\|A^{\frac{1}{2}}a\|^{\beta}Aa=\xi,
 $$
 where $\|\xi\|\neq 0$. \\
 We set  $\|u'\|^{l}u'=\psi(t)$, then $\psi(0)=0$. For any $t\neq 0$ we have
  $$
  \frac{\psi(t)}{t}=\frac{1}{t}\Int_0^t\psi'(s)\,ds.
  $$
  It follows that $\Frac{\psi(t)}{t}\rightarrow \xi$ as  $t\rightarrow 0$, therefore for  $|t|$ small enough we have,
  $$
  \|\psi(t)\|\geq \frac{t}{2}\|\xi\|.
  $$
 Hence, $\|u'\|^{l}\geq \eta t^{\frac{l}{l+1}}$
  for $\vert t\vert$ small enough and some $\eta>0$.\\

 Integrating (\ref{1}) over $(0,t)$, we have since $u'(0) = 0$
 $$
 \|u'(t)\|^{l}u'(t) = -\Int_0^t\|A^{\frac{1}{2}}u(\tau)\|^{\beta}Au(\tau)\,d\tau-\Int_0^t g(u'(\tau))\,d\tau.
 $$
 Let $u(t)$ and $v(t)$ be two solutions, then $w(t)=u(t)-v(t)$ satisfies
 {\small\begin{equation}\label{hh}
\begin{split}
 \|u'(t)\|^{l}u'(t) -\|v'(t)\|^{l}v'(t) &= -\Int_0^t(\|A^{\frac{1}{2}}u(\tau)\|^{\beta}Au(\tau)-\|A^{\frac{1}{2}}v(\tau)\|^{\beta}Av(\tau))\,d\tau\\&-
 \Int_0^t (g(u'(\tau))-g(v'(\tau)))\,d\tau.
 \end{split}
\end{equation}}
Applying Corollary \ref{coerc} with $f(s)= s^l $ and $p=l$ , we get
$$
\|\|u'(t)\|^{l}u'(t)- \|v'(t)\|^{l}v'(t)\|\geq \eta t^{\frac{l}{l+1}}\|u'(t)-v'(t)\|,
$$
and applying Lemma \ref{LEM12}, (\ref{SS1}) and from (\ref{hh}), we now deduce
{\small\begin{equation}\label{hhm1h}
\begin{split}
 \|w'(t)\|&\leq\frac{C}{t^{\frac{l}{l+1}}}\Int_0^t\Int_0^\tau \|w'(s)\|\, ds\,d\tau+
 \frac{C}{t^{\frac{l}{l+1}}}\Int_0^t\|w'(\tau)\|\,d\tau\\&\leq
  C(T)t^{-\frac{l}{l+1}}\Int_0^t \|w'(\tau)\|\,d\tau.
 \end{split}
\end{equation}}
Setting  $\phi(t)=\Int_0^t\|w'(\tau)\|\,d\tau$, by solving \eqref{hhm1h} on $[\delta, t]$ we obtain
$$
\phi(t)\leq \phi(\delta)e^{C(T)\int_\delta^t s^{-l/(l+1)}\,ds},
$$
and by letting $\delta \rightarrow 0$ we conclude that  $w(t)=0.$ A similar argument gives the uniqueness for $t$ negative with $\vert t\vert$ small enough.
\end{proof}
\begin{proposition}\label{PRO7}
For any interval $J$ and any  $t_0\in J$ if a solution $u$ of (\ref{1}) satisfies
$$
u\in {\mathcal C}^1(J,H),\,\,\,\,\,\,\|u'\|^{l}u'\in {\mathcal C}^1(J,H)\,\,\,\,\
\mbox{and}\,\,\,\, u(t_0)= u'(t_0)=0,
$$
then $u\equiv 0$.
\end{proposition}
\begin{proof} From theorem 2.3 we know that 
\begin{equation*}
\frac{d}{dt}E(t)=-(g(u'(t)),u'(t)).
\end{equation*}
Using (\ref{SS1}), we have
$$ |\frac{d}{dt}E(t)|=|(g(u'(t)),u'(t))|\leq k_3\|u'\|^{l+2}\leq k_3E(t). $$
Now, let $t_0 \in J$ such that $E(t_0)=0$. By integration  we get
$$|E(t)|\leq |E(t_0)|e^{k_3|t-t_0| }= 0.$$ \end{proof}
\section {Energy estimates  for equation (\ref{1})}
In this section, we suppose that
 \begin{equation}\label{QA11}
\exists \eta_1> 0,\,\,\,\,\,\, \forall v,\,\,\,  \|g(v)\|\leq \eta_1 \|v\|^{\alpha+1},
\end{equation}
and
\begin{equation}\label{QA1}
\exists \eta_2> 0,\,\,\,\,\,\, \forall v,\,\,\, (g(v),v)\geq \eta_2\|v\|^{\alpha+2},
\end{equation}
for some $\alpha > 0$.
\begin{theorem}\label{Th2}
Assuming $\alpha> l$, there exists  a positive constant $\eta$ such that if $u$ is any solution of (\ref{1}) with $E(0)\neq0$
     \begin{equation}\label{XE}
\liminf_{t\rightarrow +\infty}t^{\frac{l+2}{\alpha-l}}E(t)\geq \eta.
\end{equation}
\begin{itemize}
 \item[(i)] If $\alpha \geq \frac{\beta(1+l)+l}{\beta+2}$,
     then there is a constant $C(E(0))$ depending on
     $E(0)$ such that
     $$
     \forall t\geq 1,\,\,\,\ E(t)\leq C(E(0)) t^{-\frac{l+2}{\alpha-l}}.
     $$
     \item[(ii)] If $\alpha< \frac{\beta(1+l)+l}{\beta+2}$ ,
         then there is a constant $C(E(0))$ depending
         on $E(0)$ such that
     $$
     \forall t\geq 1,\,\,\,\ E(t)\leq C(E(0)) t^{-\frac{(\alpha+1)(\beta+2)}
     {\beta-\alpha}}.
     $$
 \end{itemize}
\end{theorem}
\begin{proof}
From the definition of $E(t)$ we have
$$
\|u'(t)\|^{\alpha+2}\leq C({ l},\alpha)E(t)^{\frac{\alpha+2}{l+2}},
$$
where $C({ l},\alpha)$ is a positive constant, hence from (\ref{V1}) and (\ref{QA11}) we deduce
$$
\frac{d}{dt}E(t)\geq -C({ l},\alpha,\eta_1)E(t)^{\frac{\alpha+2}{l+2}}.
$$
Assuming $\alpha>l$ we derive
{\small\begin{equation*}
 \begin{split}
 \frac{d}{dt}E(t)^{-\frac{\alpha-l}{l+2}}&=-\frac{\alpha -l}{l+2}E'(t)E(t)^{-\frac{\alpha+2}{l+2}}\\&
 \leq \frac{\alpha -l}{l+2}C({ l},\alpha,\eta_1)=C_1.
 \end{split}
\end{equation*}}
By integrating, we get
$$
E(t)\geq (E(0)^{-\frac{\alpha-l}{l+2}}+C_1t)^{-\frac{l+2}{\alpha-l}},
$$
implying
$$
\liminf_{t\rightarrow +\infty}t^{\frac{l+2}{\alpha-l}}E(t)\geq \eta =C_1^{-\frac{l+2}{\alpha-l}}.
$$
Hence (\ref{XE}) is proved. Now, we show $(i)$ and $(ii)$,
we consider the perturbed energy function
\begin{equation}\label{4}
E_\varepsilon (t) = E(t) + \varepsilon  (\|u\|^{2\gamma}u,\|u'\|^{l}u'),
\end{equation}
where ${ l} > 0$,\,\,$\gamma >0$ and $\varepsilon > 0.$\\
By Young's inequality, with the conjugate exponents $l+2$ and $\Frac{l+2}{l+1}$, we get
{\small\begin{equation*}
\begin{split}
\vert (\|u\|^{2\gamma}u,\|u'\|^{l} u')\vert &
\leq C_1\|A^{\frac{1}{2}}u\|^{(2\gamma+1)(l+2)}+c_2\|u'\|^{l+2}.
\end{split}
\end{equation*}}
We choose $\gamma$ so that $(2\gamma+1)(l+2)\geq \beta +2$, which reduces  to
\begin{equation}\label{55}
\gamma \geq \Frac{\beta -l}{2(l +2)}.
\end{equation}
Then, for some $C_1> 0, M>0$, we have
{\small\begin{equation}\label{5}
\begin{split}
\vert (\|u\|^{2\gamma}u,\|u'\|^{l} u')\vert &\leq C_1\|A^{\frac{1}{2}}u\|^{\beta+2}+c_2\|u'\|^{l+2}\\&\leq ME(t).
\end{split}
\end{equation}}
By using (\ref{5}), we obtain from (\ref{4})
$$
(1-M\varepsilon)E(t)\leq E_\varepsilon(t) \leq (1+M\varepsilon)E(t).
$$
Taking  $\varepsilon\leq \Frac{1}{{ 2M}}$, we deduce
\begin{equation}\label{6}
\forall t\geq 0,\,\,\, \Frac{1}{2}E(t)\leq E_\varepsilon(t)\leq 2 E(t).
\end{equation}
On the other hand, we have
{\small\begin{equation*}
\begin{split}
E'_\varepsilon (t) &= -(g(u'),u')+ \varepsilon  ((\|u\|^{2\gamma} u)',\|u'\|^{l} u')+
\varepsilon  (\|u\|^{2\gamma} u ,(\|u'\|^{l} u')').
\end{split}
\end{equation*}}
We observe that
\begin{equation*}
\begin{split}
(\|u\|^{2\gamma}u)'&=((\|u\|^{2})^{\gamma})'u+\|u\|^{2\gamma}u'
\\&
=2\gamma\|u\|^{2(\gamma-1)}( u ,u')u +\|u\|^{2\gamma}u',
\end{split}
\end{equation*}
and we have
\begin{equation*}
\begin{split}
((\|u\|^{2\gamma}u)',\|u'\|^{l} u')&=2\gamma\|u\|^{2(\gamma-1)}\|u'\|^{l}(u,u) ( u', u') +(\|u\|^{2\gamma}u',\|u'\|^{l} u')\\&=
2\gamma\|u\|^{2\gamma}\|u'\|^{l+2}+(\|u\|^{2\gamma}u',\|u'\|^{l} u').
\end{split}
\end{equation*}
Then, we can deduce that 
{\small\begin{equation}\label{33aH9}
\begin{split}
E'_\varepsilon (t) &= -(g(u'),u')+ \varepsilon 2\gamma\|u\|^{2\gamma}\|u'\|^{l+2} +\varepsilon (\|u\|^{2\gamma}u',\|u'\|^{l} u')\\&
-\varepsilon  (\|u\|^{2\gamma}u,\|A^{\frac{1}{2}}u\|^{\beta}Au)-\varepsilon  (\|u\|^{2\gamma}u,g(u')).
\end{split}
\end{equation}}
We now estimate the right side of (\ref{33aH9}),\\
The fourth term:
{\small\begin{equation}\label{33aH}
\begin{split}
-\varepsilon  (\|u\|^{2\gamma}u,\|A^{\frac{1}{2}}u\|^\beta Au)&=-  \varepsilon  \|u\|^{2\gamma}  
\|A^{\frac{1}{2}} u\|^\beta(u, A u)\\&=- \varepsilon  \|u\|^{2\gamma}  \|A^{\frac{1}{2}} u\|^{\beta+2}\\&\leq
 -c \varepsilon \|A^{\frac{1}{2}} u\|^{2\gamma +\beta+2}.
\end{split}
\end{equation}}
The second term:
\begin{equation*}
\begin{split}
\|u\|^{2\gamma}\|u'\|^{l+2} \leq C_2\|A^{\frac{1}{2}}u\|^{2\gamma}\|u'\|^{l+2}.
\end{split}
\end{equation*}
The third term:
\begin{equation*}
\begin{split}
|(\|u\|^{2\gamma}u',\|u'\|^{l} u')|=\|u\|^{2\gamma}\|u'\|^{l} (u',u')\leq
C_3 \|A^{\frac{1}{2}}u\|^{2\gamma}\|u'\|^{l+2}.
\end{split}
\end{equation*}
Applying Young's inequality, with the conjugate exponents $\frac{\alpha+2}{\alpha-l}$ and $\frac{\alpha+2}{l+2}$,
we have
$$
\|A^{\frac{1}{2}}u\|^{2\gamma}\|u'\|^{l+2}\leq
\delta\|A^{\frac{1}{2}}u\|^{2\gamma\frac{\alpha+2}{\alpha-l}}+C(\delta) \|u'\|^{(l+2)\frac{\alpha+2}{l+2}}.
$$
We assume $$\frac{(\alpha+2)2\gamma}{\alpha-l}\geq 2\gamma+\beta
+2,$$ which reduces to the condition
\begin{equation}\label{555}
 \gamma \geq
\frac{(\beta+2)(\alpha-l)}{2(l+2)},
\end{equation}
and taking $\delta$ small enough, we have for some $P> 0$, therefore the second and the third terms becomes
\begin{equation}\label{33aH1}
\begin{split}
\varepsilon 2\gamma\|u\|^{2\gamma}\|u'\|^{l+2} +\varepsilon (\|u\|^{2\gamma}u',\|u'\|^{l} u')&\leq  \varepsilon C ( 2\gamma +1)\|A^{\frac{1}{2}}u\|^{2\gamma}\|u'\|^{l+2} 
\\&\leq \frac{\varepsilon }{4}\|A^{\frac{1}{2}}u\|^{2\gamma+\beta+2}+\varepsilon P \|u'\|^{\alpha+2}. 
\end{split}
\end{equation}
Using (\ref{QA1}), (\ref{33aH9}),  (\ref{33aH}) and (\ref{33aH1}), we have
{\small\begin{equation}\label{33aH4}
\begin{split}
E'_\varepsilon (t) \leq (-\eta_2+\varepsilon P ) \|u'\|^{\alpha+2}-\varepsilon  \|A^{\frac{1}{2}}u\|^{2\gamma+\beta+2}+
\frac{\varepsilon }{4}\|A^{\frac{1}{2}}u\|^{2\gamma+\beta+2}
-\varepsilon  (\|u\|^{2\gamma}u,g(u')).
\end{split}
\end{equation}}
Applying Young's inequality, with the conjugate exponents $\alpha+2$ and
$\frac{\alpha+2}{\alpha+1}$ and using (\ref{QA11}),  we have
{\small\begin{equation*}
\begin{split}
- (\|u\|^{2\gamma}u,g(u'))&
\leq \delta(\|u\|^{(2\gamma+1)(\alpha+2)}
+C'(\delta)\|g(u')\|^{\frac{\alpha+2}{\alpha+1}}
\\&\leq C_4\delta\|A^{\frac{1}{2}}u\|^{(2\gamma+1)(\alpha+2)}
+C'(\delta)\|u'\|^{\alpha+2}.
\end{split}
\end{equation*}}
This term will be dominated by the negative terms assuming
$$(2\gamma+1)(\alpha+2)\geq 2\gamma+\beta+2\Leftrightarrow
(\alpha+1)(2\gamma+1)\geq \beta+1.$$ This is equivalent to the condition
\begin{equation}\label{5555M}
\gamma \geq \frac{\beta-\alpha}{2(\alpha+1)},
\end{equation}
and taking $\delta$ small enough, we have
$$
-\varepsilon(\|u\|^{2\gamma}u,g(u'))\leq  \frac{\varepsilon }{4}\|A^{\frac{1}{2}}u\|^{2\gamma+\beta+2}
+P'\varepsilon\|u'\|^{\alpha+2}.
$$
By replacing in (\ref{33aH4}), we have
$$
E'_\varepsilon(t)\leq(-\eta_2+Q\varepsilon)\|u'\|^{\alpha+2}-\Frac{\varepsilon }{2}\|A^{\frac{1}{2}}u\|^{2\gamma+\beta+2},
$$
where $Q=P+P'$. By choosing $\varepsilon$ small, we get
{\small\begin{equation}\label{0}
 \begin{split}
E'_\varepsilon(t)&\leq-\frac{\varepsilon}{2}\Big(\|u'\|^{\alpha+2}
+\|A^{\frac{1}{2}}u\|^{2\gamma+\beta+2}\Big)\\&\leq -\frac{\varepsilon}{2}
\Big((\|u'\|^{l+2})^{\frac{\alpha+2}{l+2}}
+(\|A^{\frac{1}{2}}u\|^{\beta+2})^{\frac{2\gamma+\beta+2}{\beta+2}}\Big).
\end{split}
\end{equation}}
This inequality will be satisfied under the assumptions (\ref{55}), (\ref{555}) and (\ref{5555M}) which lead to the sufficient condition
\begin{equation}
\gamma \geq \gamma_0 = \max\Big\{\frac{\beta-l}{2(l+2)},\frac{(\beta+2)(\alpha-l)}{2(l+2)},
\frac{\beta-\alpha}{2(\alpha+1)}\Big\}.
\end{equation}
We now  distinguish $2$ cases.

\begin{itemize}

 \item[(i)] If $\alpha \geq \Frac{\beta(1+l)+l}{\beta +2}$,
     then clearly $\Frac{(\beta+2)(\alpha-l)}{2(l+2)} \geq
     \Frac{\beta-l}{2(l+2)}.$
     
Moreover
$$
\Frac{\beta-\alpha}{2(\alpha+1)}=\frac{1}{2}\left(\Frac{\beta+1}{\alpha+1}-1\right) \leq
\frac{1}{2}\left(\Frac{\beta+1}{\frac{\beta(1+l)+l}{\beta+2}+1}-1\right) = \Frac{\beta-l}{2(l+2)}.
$$
In this case $\gamma_0=\Frac{(\beta+2)(\alpha-l)}{2(l+2)}$ and
choosing $\gamma=\gamma_0$, we find
$$
2\gamma+\beta+2 = \Frac{\alpha+2}{l+2}(\beta+2),
$$
since $ \Frac{2\gamma+\beta+2}{\beta+2} =1+
\Frac{\alpha-l}{l+2}$, replacing in (\ref{0}), we obtain for some $\rho>0$
\begin{equation*}
\begin{split}
E'_\varepsilon(t)\leq
 -\rho E(t)^{1+\frac{\alpha-l}{l+2}}\leq -\rho'
E_\varepsilon(t)^{1+\frac{\alpha-l}{l+2}},
\end{split}
\end{equation*}
where $\rho$ and $\rho'$ are positive constants.
\item[(ii)] If $ \alpha< \Frac{\beta(l+1)+l}{\beta+2}$, then
     $\Frac{(\beta+2)(\alpha-l)}{l+2}<\Frac{\beta-l}{l+2}.$
    
Moreover
\begin{eqnarray*}
\Frac{\beta-\alpha}{2(\alpha+1)}-\Frac{\beta-l}{2(l+2)}
&=& \Frac{(\beta-\alpha)(l+2)-(\beta-l)(\alpha+1)}{2(\alpha+1)(l+2)}\\
&=&\Frac{\beta(l+1)+l-\alpha(\beta+2)}{2(\alpha+1)(l+2)}> 0.
\end{eqnarray*}
In this case $\gamma_0= \Frac{\beta-\alpha}{2(\alpha+1)}$ and
choosing $\gamma= \gamma_0,$ we find
\begin{equation}\label{567B}
2\gamma+\beta+2 =(\beta+2)\left(1+\Frac{2\gamma}{\beta+2}\right)
=(\beta+2)\left(1+\Frac{\beta-\alpha}{(\alpha+1)(\beta+2)}\right),
\end{equation}
since $\gamma > \Frac{(\beta+2)(\alpha-l)}{l+2}$, we have
$$
\Frac{2\gamma+\beta+2}{\beta+2}=1 + \Frac{2\gamma}{\beta+2}> 1+\Frac{\alpha-l}{l+2}
=\Frac{\alpha+2}{l+2},
$$
replacing in (\ref{0}), we obtain {\small\begin{equation*}
 \begin{split}
E'_\varepsilon(t)&\leq -\delta\varepsilon (\|u'\|^{l+2}
+\|A^{\frac{1}{2}}u\|^{\beta+2})^\frac{2\gamma+\beta+2}{\beta+2},
\end{split}
\end{equation*}}
for some $\delta> 0$. Using (\ref{567B}), we have
\begin{equation*}
E'_\varepsilon(t)\leq -\rho E^{(1+\frac{\beta-\alpha}{(\alpha+1)(\beta+2)})}
\leq-\rho' E_\varepsilon(t)^{(1+\frac{\beta-\alpha}{(\alpha+1)(\beta+2)})},
\end{equation*}
where $\rho$ and $\rho'$ are positive constants.
\end{itemize}
\end{proof}
\section{Slow and fast solutions for equation (\ref{1}) }

In the case where  $\alpha< \frac{\beta(1+l)+l}{\beta+2}$, Theorem \ref {Th2} gives two different decay rates for the lower and the upper estimates of the energy. In the scalar case, this fact was explained in \cite{AMH} by the existence of two (and only two) different decay rates of the solutions, corresponding precisely to the lower and upper estimates. The solutions behaving as the lower estimate were called ``fast solutions" and those behaving as the upper estimate were called ``slow solutions." Moreover in the scalar case, it was shown that the set of initial data giving rise to ``slow solutions" has non-empty interior in the phase space $\R^2.$ \\

In the general case, by reducing the problem to a related scalar equation, it is rather immediate to show the coexistence of slow and fast solutions in the special case of power nonlinearities. More precisely we have 

\begin{proposition} Let $\alpha< \frac{\beta(1+l)+l}{\beta+2}$ and $c>0$.  
Then the equation  $$ (\|u'\|^{l}u')'+\|A^{\frac{1}{2}}u\|^\beta Au+c\|u'\|^\alpha u' = 0 $$ has an infinity of ``fast solutions" with energy comparable to  $t^{-\frac{l+2}{\alpha-l}}$ and an infinity of ``slow solutions" with energy comparable to $ t^{-\frac{(\alpha+1)(\beta+2)}{\beta-\alpha}}$ as $t$ tends to infinity.
\end{proposition}

\begin{proof} Let $\lambda>0$ be any eigenvalue of $A$ and $A\varphi =\lambda \varphi$ with $\|\varphi\|= 1. $ Let $(v_0,v_1) \in \R^2$ and $v$ be the solution of \begin{equation*}
(|v'|^{l}v')'+C_1|v|^\beta v+C_2|v'|^\alpha v'=0,
\end{equation*} where $C_1, C_2$ are positive constants to be chosen later. Then $u(t) = v(t) \varphi$ satisfies

$$(\|u'\|^{l}u')'+\|A^{\frac{1}{2}}u\|^\beta Au+c\|u'\|^\alpha u'= (|v'|^{l}v')'\|\varphi\|^{l}\varphi+|v|^\beta v\|A^{\frac{1}{2}}\varphi\|^\beta A\varphi+c|v'|^\alpha v'\|\varphi\|^\alpha\varphi
$$ 
$$= [(|v'|^{l}v')'+\lambda^{\frac{\beta}{2} +1}|v|^\beta v +c|v'|^\alpha v']\varphi. $$  Choosing $ C_1= \lambda^{\frac{\beta}{2} +1}$ and $C_2 = c$ , $u(t) = v(t) \varphi$ becomes a solution of the vector equation. The existence of an infinity of ``fast solutions"  and an infinity of ``slow solutions" are then an immediate consequence of the same result for the scalar equation proven in  {\bf\cite{AMH}}. \end{proof}

\begin{remark}\rm{In the special case $A = \lambda I $, we can take for $\varphi$ any vector of the sphere $\|\varphi\|= 1. $ We obtain in this way an open set of slow solutions in $ \mathbb{R}^{N+1}$ corresponding to the initial data of the form $(v_0\varphi, v_1\varphi). $ Actually, in the general case, by generalizing a modified energy method introduced in \cite{GGH} , under the additional condition $l<1$ we shall now prove the existence of an open set of slow solutions. } \end{remark}

\begin{theorem} \label{Main_Theorem}
Assume that $g$ satisfies (\ref{QA1}) and $l<1$, $l<\alpha <\frac{\beta(1+l)+l}{\beta+2}$. Then, there exist a nonempty open set $\mathcal{S}\subset H\times H$ and a constant $M$ such that, for every $(u_0,u_1) \in \mathcal{S}$, the unique global solution of equation (\ref{1}) with initial data $(u_0,u_1)$ satisfies 
\begin{eqnarray}\label{estimate1}
\left\Vert u(t) \right\Vert \ge \frac{M}{(1+t)^{\frac{\alpha+1}{\beta-\alpha}}} \text{\ \ \ } \forall t \ge 0.
\end{eqnarray}
\end{theorem}
\begin{proof}
Assuming $(u_0,u_1)\in H \times H$ and $u_0\ne0$, we consider the following constants
\begin{eqnarray*}
\sigma_0:=\left(\frac{(\beta+2)(l+1)}{l+2} \left\Vert u_1\right\Vert ^{l+2}+ \|A^{\frac{1}{2}}u_0\|^{\beta+2} \right)^{\frac{1}{\beta+2}},
\end{eqnarray*}
and
\begin{eqnarray*}
\sigma_1:=\frac{\left\Vert u_1\right\Vert^2}{\|A^{\frac{1}{2}}u_0\|^{2\left(\frac{\beta-\alpha}{\alpha+1}+1\right)}}.
\end{eqnarray*}
For any $\varepsilon_0>0$, $\varepsilon_1>0$, the set $\mathcal{S}\subset H \times H$ of initial data such that $\sigma_0<\varepsilon_0$, $\sigma_1<\varepsilon_1$ is clearly a nonempty open set which contains at least all pairs $(u_0,u_1)$ with $u_1=0$ and $u_0\ne 0$ with $\left\Vert u_0 \right\Vert$ small enough. We claim that if $\varepsilon_0$ and $\varepsilon_1$ are small enough, for any $(u_0,u_1)\in \mathcal{S}$, the global solution of (\ref{1}) satisfies (\ref{estimate1}).

First of all, from (\ref{V1}) and (\ref{QA1}), we see that 
\begin{eqnarray*}
\frac{d}{dt}E(t)\leq -\eta_2\left\Vert u^{\prime}(t) \right\Vert^{\alpha+2}<0,
\end{eqnarray*}
hence $E(t)\leq E(0)$ for every $t\ge0$, and from (\ref{energy}), we deduce that 
\begin{eqnarray}\label{estimate_u}
\|A^{\frac{1}{2}}u(t)\| \leq \left[(\beta+2)E(0)\right]^{\frac{1}{\beta+2}}=\sigma_0 \text{\ \ } \forall t\ge0.
\end{eqnarray}
Then, we shall establish that if $\varepsilon_0$ and $\varepsilon_1$ are small enough, we have
\begin{eqnarray}\label{condition11}
u(t)\ne 0 \text{\ \ \ } \forall t\ge0,
\end{eqnarray}
and for some $C>0$
\begin{eqnarray}\label{condition22}
\frac{\left\Vert u^{\prime}(t)\right\Vert^2}{\|A^{\frac{1}{2}}u(t)\|^{2\left(\frac{\beta-\alpha}{\alpha+1}+1\right)}}\leq C \text{\ \ \ } \forall t\ge0.
\end{eqnarray}
Assuming this inequality, it follows immediately that $u$ is a slow solution. Indeed let us set $y(t):=\left\Vert u(t) \right\Vert^2$. We observe that
\begin{eqnarray*}\label{estimate_ne}
\left\vert y^{\prime}(t)\right\vert&=&2\left\vert \left ( u^{\prime}(t),u(t)\right )\right\vert \leq 2 \left\Vert u^{\prime}(t)\right\Vert \cdot \left\Vert u(t)\right\Vert \nonumber
\\
&\leq& 2 \frac{\left\Vert u^{\prime}(t)\right\Vert}{\|A^{\frac{1}{2}}u(t)\|^{\frac{\beta-\alpha}{\alpha+1}+1}}\cdot \|A^{\frac{1}{2}}u(t)\|^{\frac{\beta-\alpha}{\alpha+1}+1} \left\Vert u(t)\right\Vert \nonumber\\
&\leq&2\sqrt{C}\|u(t)\|^{\frac{\beta-\alpha}{\alpha+1}+2} \leq 2\sqrt{C} \left\vert y(t) \right\vert^{\frac{\beta-\alpha}{2(\alpha+1)}+1},
\end{eqnarray*}
and in particular
\begin{eqnarray*}\label{inequality_ne}
y^{\prime}(t)\ge -2\sqrt{C} \left\vert y(t) \right\vert^{\frac{\beta-\alpha}{2(\alpha+1)}+1}\text{\ \ \ } \forall t\ge0.
\end{eqnarray*}
Taking into account that
\begin{eqnarray*}
\left( y(t)^{-\frac{\beta-\alpha}{2(\alpha+1)}}\right)^{\prime}=-\frac{\beta-\alpha}{2(\alpha+1)}y(t)^{-\frac{\beta-\alpha}{2(\alpha+1)}-1}y^{\prime}(t),
\end{eqnarray*}
we have
\begin{eqnarray*}
\left( y(t)^{-\frac{\beta-\alpha}{2(\alpha+1)}}\right)^{\prime}\leq \sqrt{C}\frac{\beta-\alpha}{\alpha+1}.
\end{eqnarray*}
Integrating between $0$ and $t$ and since $y(0)>0$, we deduce that there exists a constant $M_1$ such that
\begin{eqnarray*}
 y(t)^{-\frac{\beta-\alpha}{2(\alpha+1)}}\leq M_1 (t+1) \text{\ \ \ } \forall t\ge0.
\end{eqnarray*}
This inequality concludes the proof. So we are left to prove (\ref{condition11}) and (\ref{condition22}). 

Let us set
\begin{eqnarray*}
T:=\sup \{t\ge0:\forall \tau\in [0,t], \text{\ \ }u(\tau)\ne 0\}.
\end{eqnarray*}
Since $u(0)\ne 0$, we have that $T>0$ and if $T<+\infty$, then $u(T)=0$.

Let us consider, for all $t\in(0,T)$, the energy
\begin{eqnarray}\label{definition_H}
H(t):=\frac{\left\Vert u^{\prime}(t)\right\Vert^{2}}{\|A^{\frac{1}{2}}u(t)\|^{2\gamma}},
\end{eqnarray}
where $\gamma:=\displaystyle \frac{\beta-\alpha}{\alpha+1}+1$.

If $u'(t) = 0$, then we claim that $H$ is differentiable at $t$ with $H'(t) = 0$. Indeed by the equation $$ \Big(\|u'\|^{l}u'\Big)'+\|A^{\frac{1}{2}}u\|^\beta Au+g(u')=0,
$$ we find that $\|u'\|^{l}u'(s) = O(s-t)$ for $s$ close to $t$, then  since $l<1$ we deduce $\left\Vert u^{\prime}(t)\right\Vert^{2} = O(s-t)$, thereby proving the claim. \\

If $u'(t) \not= 0$ , then again  $H$ is differentiable at $t$ with \begin{eqnarray*}
H^{\prime}(t)=\frac{\displaystyle \frac{d}{dt} \left(\left\Vert u^{\prime}(t) \right\Vert^{2+l}\right)}{\|A^{\frac{1}{2}}u(t)\|^{2\gamma}\left\Vert u^{\prime}(t) \right\Vert^{l}}-\frac{\displaystyle \frac{d}{dt} \left(\|A^{\frac{1}{2}}u(t)\|^{2\gamma}\left\Vert u^{\prime}(t) \right\Vert^{l}\right)}{\|A^{\frac{1}{2}}u(t)\|^{2\gamma}\left\Vert u^{\prime}(t) \right\Vert^{l}}H(t).
\end{eqnarray*}
Taking into account (\ref{energy}) and (\ref{V1}), we deduce 
\begin{eqnarray}\label{H_prime}
H^{\prime}(t)&=&-\rho \frac{(g(u^{\prime}(t)), u^{\prime}(t))}{\|A^{\frac{1}{2}}u(t)\|^{2\gamma}\left\Vert u^{\prime}(t) \right\Vert^{l}}-\rho \frac{ \|A^{\frac{1}{2}}u(t)\| ^{\beta-2\gamma}\left ( u^{\prime}(t),Au(t)\right )}{\left\Vert u^{\prime}(t) \right\Vert^{l}} \nonumber \\
&&-\frac{\displaystyle \frac{d}{dt} \left(\|A^{\frac{1}{2}}u(t)\|^{2\gamma}\left\Vert u^{\prime}(t) \right\Vert^{l}\right)}{\|A^{\frac{1}{2}}u(t)\|^{2\gamma}\left\Vert u^{\prime}(t) \right\Vert^{l}}H(t)=:H_1+H_2+H_3,
\end{eqnarray}
where $\rho:=\displaystyle \frac{l+2}{l+1}$.
\\

Let us estimate $H_1$, $H_2$ and $H_3$. Using (\ref{QA1}) and (\ref{definition_H}), we observe that
\begin{eqnarray}\label{I1L}
H_1&\leq&-\rho\, \eta_2\frac{\left\Vert u^{\prime}(t)\right\Vert^{\alpha+2-l}}{\|A^{\frac{1}{2}}u(t)\|^{2\gamma}}=- \rho \, \eta_2\left(\frac{\left\Vert u^{\prime}(t)\right\Vert^2}{\|A^{\frac{1}{2}}u(t)\|^{2\gamma}} \right)^{\frac{\alpha+2-l}{2}}\frac{\|A^{\frac{1}{2}}u(t)\|^{\gamma(\alpha+2-l)}}{\|A^{\frac{1}{2}}u(t)\|^{2\gamma}} \nonumber\\
&=&-\rho \, \eta_2 H^{\frac{\alpha+2-l}{2}}(t)\|A^{\frac{1}{2}}u(t)\|^{\gamma( \alpha-l)}.
\end{eqnarray}
In order to estimate $H_2$, we use Young's inequality applied with the conjugate exponents $\frac{\alpha+2-l}{1-l}$ and $\frac{\alpha+2-l}{\alpha+1}$,
\begin{eqnarray*}
\left\vert H_2 \right\vert &\leq& \rho\left\Vert u^{\prime}(t)\right\Vert^{1-l} \cdot \frac{\|A^{\frac{1}{2}}u(t)\|^{\beta+1}}{\|A^{\frac{1}{2}}u(t)\|^{2\gamma}}=\rho \frac{\left\Vert u^{\prime}(t)\right\Vert^{1-l}}{\|A^{\frac{1}{2}}u(t)\|^{\frac{2\gamma(1-l)}{\alpha+2-l}}}\cdot \frac{\|A^{\frac{1}{2}}u(t)\|^{\frac{2\gamma(1-l)}{\alpha+2-l}+\beta+1}}{\|A^{\frac{1}{2}}u(t)\|^{2\gamma}}\\
&\leq& \delta_1 \left(  \frac{\left\Vert u^{\prime}(t)\right\Vert^{1-l}}{\|A^{\frac{1}{2}}u(t)\|^{\frac{2\gamma(1-l)}{\alpha+2-l}}} \right)^\frac{\alpha+2-l}{1-l}+\delta_2 \left( \|A^{\frac{1}{2}}u(t)\|^{\frac{2\gamma(1-l)}{\alpha+2-l}+\beta+1-2\gamma}\right)^{\frac{\alpha+2-l}{\alpha+1}},
\end{eqnarray*}
where $\delta_1:=\displaystyle \rho\, \eta_2\left( \frac{1-l}{\alpha+2-l}\right)$ and $\delta_2:= \displaystyle \frac{\rho}{\eta_2}\left( \frac{\alpha+1}{\alpha+2-l}\right)$, and taking into account that $\gamma=\frac{\beta+1}{\alpha+1}$, we deduce
\begin{eqnarray}\label{I2}
\left\vert H_2 \right\vert &\leq&\delta_1 \left(  \frac{\left\Vert u^{\prime}(t)\right\Vert^2}{\|A^{\frac{1}{2}}u(t)\|^{2\gamma}} \right)^{\frac{\alpha+2-l}{2}}\frac{\|A^{\frac{1}{2}}u(t)\|^{\gamma(\alpha+2-l)}}{\|A^{\frac{1}{2}}u(t)\|^{2\gamma}}+\delta_2 \|A^{\frac{1}{2}}u(t)\| ^{\gamma( \alpha-l)} \nonumber\\
&=& \delta_1H^{\frac{\alpha+2-l}{2}}(t)\|A^{\frac{1}{2}}u(t)\|^{\gamma (\alpha-l)}+\delta_2 \|A^{\frac{1}{2}}u(t)\| ^{\gamma( \alpha-l)}.
\end{eqnarray}
{ For simplicity in the sequel we shall write $$ l/2: = m>0.$$
}
In order to estimate $H_3$, we use (\ref{definition_H}) and we have that
\begin{eqnarray*}
H_3&=&-\frac{\displaystyle \frac{d}{dt} \left(\|A^{\frac{1}{2}}u(t)\|^{2\gamma(1+m)}H(t)^m\right)}{\|A^{\frac{1}{2}}u(t)\|^{2\gamma}\left\Vert u^{\prime}(t) \right\Vert^{l}}H(t)=-mH(t)^mH^{\prime}(t)\frac{\|A^{\frac{1}{2}}u(t)\|^{2\gamma m}}{\left\Vert u^{\prime}(t) \right\Vert^{l}}\\
&&- 2\gamma (1+m)H(t)^{m+1}\frac{\|A^{\frac{1}{2}}u(t)\|^{2(\gamma m-1)}}{\left\Vert u^{\prime}(t) \right\Vert^{l}}\left ( u^{\prime}(t),Au(t)\right )=:H_4+H_5.
\end{eqnarray*}
We observe that taking into account (\ref{definition_H}), we have
\begin{eqnarray}\label{I3_1}
H_4=-mH^{\prime}(t),
\end{eqnarray}
and
\begin{eqnarray*}
\left\vert H_5 \right\vert \leq 2 \gamma(1+m)H(t)^{m+1}\frac{\|A^{\frac{1}{2}}u(t)\|^{2\gamma m}}{\left\Vert u^{\prime}(t)\right\Vert^{l}}\cdot \frac{\left\Vert u^{\prime}(t)\right\Vert}{\|A^{\frac{1}{2}}u(t)\|}=2\gamma(1+m) H(t)^{\frac{3}{2}}\|A^{\frac{1}{2}}u(t)\| ^{\gamma-1}.
\end{eqnarray*} 
We observe that the hypothesis \begin{eqnarray*}
\alpha <\frac{\beta(1+l)+l}{\beta+2},
\end{eqnarray*}
implies
\begin{eqnarray*}
\gamma-1> \gamma(\alpha-l), \end{eqnarray*}
and taking into account (\ref{estimate_u}), we obtain
\begin{eqnarray}\label{I3_2}
\left\vert H_5 \right\vert \leq 2\gamma(1+m) \sigma_0^{\gamma-1-\gamma(\alpha-l)}H(t)^{\frac{3}{2}}\|A^{\frac{1}{2}}u(t)\| ^{\gamma (\alpha-l)}.
\end{eqnarray} 
Then, taking into account (\ref{I1L})-(\ref{I3_2}) in (\ref{H_prime}), we deduce that
\begin{eqnarray}\label{derivative_HH}
\!\!\!\!\!\!\!\!\!\!\!\!H^{\prime}(t)&\!\!\leq\!\!& \|A^{\frac{1}{2}}u(t)\| ^{\gamma (\alpha-l)} \left[-\widehat{\rho} \, \eta_2 H^{\frac{\alpha+2-l}{2}}(t)+\frac{\widehat{\rho}}{\eta_2}+2\gamma \sigma_0^{\gamma-1-\gamma(\alpha-l)}H(t)^{\frac{3}{2}}  \right]\!,
\end{eqnarray}
where $\widehat{\rho}:=\frac{\rho(\alpha+1)}{(1+m)(\alpha+2-l)}$. \\

Here we observe that \eqref {derivative_HH} is still valid if $u'(t) = 0$, since then $H(t) = 0 = H'(t)$ and the RHS is non-negative. Now, let $$h(s,\Gamma)=-\widehat{\rho}\, \eta_2s^{\frac{\alpha+2-l}{2}}+\frac{\widehat{\rho}}{\eta_2}+ 2\gamma\Gamma s^{\frac{3}{2}},$$ where $\Gamma:= \sigma_0^{\gamma-1-\gamma(\alpha-l)}.$

 We observe that $h((2/\eta_2^2)^{2/(\alpha+2-l)},\Gamma)=-\Frac{\widehat{\rho}}{\eta_2}+ 2\gamma\Gamma \left(2/\eta_2^2\right)^{\frac{3}{\alpha+2-l}}.$ In particular $$h\left((2/\eta_2^2)^{2/(\alpha+2-l)},\Gamma\right)\rightarrow -\frac{\widehat{\rho}}{\eta_2}<0, \text{\ \ if  } \Gamma \rightarrow 0.$$Let us therefore assume that $\sigma_0$ is sufficiently small to achieve $$h((2/\eta_2^2)^{2/(\alpha+2-l)}, \sigma_0^{\gamma-1-\gamma(\alpha-l)})<0.$$

 We claim that if  $H(0)=\sigma_1<\varepsilon_1 : =(2/\eta_2^2)^{2/(\alpha+2-l)} $ , then $H(t)$ is bounded for all $t\in (0,T)$. Indeed if  $H$ is not bounded for all $t\in(0,T)$, then there exists $\bar{T}\in(0,T)$ such that $H(\bar{T})=\varepsilon_1 $ and $H(t)< \varepsilon_1 $ for all $t\in (0,\bar{T})$. By (\ref{derivative_HH}) and taking into account that $u\in \mathcal{C}^1(\mathbb{R}^+,H)$ and $\left\Vert u^{\prime}\right\Vert ^{l}u^{\prime}\in \mathcal{C}^1(\mathbb{R}^+,H) $, we have $$H^{\prime}(\bar{T})\leq \|A^{\frac{1}{2}}u(\bar{T})\| ^{\gamma( \alpha-l)}h(\varepsilon_1 ,\sigma_0^{\gamma-1-\gamma(\alpha-l)})\leq \left\Vert u(\bar{T}) \right\Vert ^{\gamma( \alpha-l)}h(\varepsilon_1 ,\sigma_0^{\gamma-1-\gamma(\alpha-l)})<0,$$
then since   $\bar{T}\in(0,T)$ implies $\left\Vert u(\bar{T}) \right\Vert>0$,  $H$ decreases near $\bar{T}$. This contradicts the definition of $\bar{T}$ and we can claim that $H$ is bounded for all $t\in (0,T)$.

Finally, we claim that $T=+\infty$. Let us assume by contradiction that this is not the case. Then, taking into account that $H$ is bounded for all $t\in [0,T)$, we observe that
\begin{eqnarray*}
\left\Vert u^{\prime}(t)\right\Vert^{2}\leq C \|A^{\frac{1}{2}}u(t)\|^{2\gamma} \leq C\left\Vert u(t)\right\Vert^{2\gamma} \text{\ \ \ } \forall t\in[0,T).
\end{eqnarray*}

Therefore, from the continuity of the vector $\left(u^{\prime},u\right)$ with values in $H \times H$ it now follows that $u(T)=0$ implies $u^{\prime}(T)=0$, hence $u\equiv 0$ by backward uniqueness, a contradiction. Hence, $T=+\infty$, we obtain (\ref{condition11}) and (\ref{condition22}), and the conclusion follows.

\end{proof}  

Our next result is the generalization to the vector case of the slow-fast alternative established in \cite{AMH} for the scalar equation. By relying on a technique introduced by Ghisi, Gobbino and Haraux \cite{GhGoHa}, we prove that all non-zero solutions to (\ref{1})  are either slow solutions or fast solutions. 

\begin{theorem}[The slow-fast alternative]\label{Main_Theorem2}
Assume that $g$  is locally Lipschitz continuous and satisfies (\ref{QA1}) and $l<1$, $l<\alpha <\frac{\beta(1+l)+l}{\beta+2}$.  Let $u(t)$ be a global solution to (\ref{1}). Then, one and only one of the following statements apply: 
 \begin{enumerate}
  \item (Fast solutions) There exist some positive constants $\eta$, $C$ and $T$ such that  
  \begin{equation*}
  \|u'(t)\|^{2} \ge  \eta \|A^{\frac{1}{2}}u(t)\|^{2\left(\frac{\beta +2}{l+2}\right)},\,\,\,\, \, \forall t\geq T,
  \end{equation*}
  and
 \begin{equation}\label{fast}
 E(t) \leq \frac{C}  {(t+1)^{\frac{l+2}{\alpha-l}}},\,\,\,\, \forall t\geq 0.
 \end{equation}
\item (Slow solutions) There exist some positive constants $M$, $c$ and $T$ such that 
\begin{equation*}
  \|u'(t)\|^{2} \leq  M \|A^{\frac{1}{2}}u(t)\|^{2\left(\frac{\beta +2}{l+2}\right)},\,\,\,\, \, \forall t\geq T,
  \end{equation*}
  and
  \begin{equation}\label{slow}
  \|u(t)\|  \geq \frac{c}{(1+t)^{\frac{\alpha+1}{\beta-\alpha}}},\,\,\,\, \, \forall t\geq 0.
 \end{equation}
 \end{enumerate}
 
\end{theorem}
\begin{proof}
{\bf Fast solutions:} 

We first establish that if for some $\eta>0$ we have 
\begin{equation}\label{H_H3}
 \|u'(t)\|^{2} \ge  \eta \|A^{\frac{1}{2}}u(t)\|^{2\left(\frac{\beta +2}{l+2}\right)},\,\,\,\, \, \forall t\geq T,
\end{equation}
then (\ref{fast}) holds.

Using (\ref{energy}), we obtain
\begin{eqnarray*} \frac{l+1}{l+2}\|u'(t)\|^{l+2}=-\frac{1}{\beta+2}\|A^{\frac{1}{2}}u(t)\|^{\beta+2}+E(t)\ge E(t)-\delta\|u'(t)\|^{l+2},\end{eqnarray*} 
where $\delta:=\left[(\beta+2)\eta^{\frac{l+2}{2}}\right]^{-1}$. Then
\begin{eqnarray*} \|u'(t)\|^{\alpha+2}\ge \left(\frac{l+1}{l+2}+\delta \right)^{-\frac{\alpha+2}{l+2}}E(t)^{\frac{\alpha+2}{l+2}}.\end{eqnarray*} 
From the energy identity, we deduce
\begin{eqnarray*} E^{\prime}(t)\leq -\widehat{C} E(t)^{\frac{\alpha-l}{l+2}+1},\end{eqnarray*} 
where $\widehat{C}:=\eta \left(\Frac{l+1}{l+2}+\delta \right)^{-\frac{\alpha+2}{l+2}}$. Integrating between $T$ and $t$ and since $E(t)\ge0$, we deduce that there exists a constant $C_1$ such that
\begin{eqnarray*} E(t)\leq C_1 (t-T)^{-\frac{l+2}{\alpha-l}},\,\,\,\, \forall t\ge T+1.\end{eqnarray*} The result follows immediately. \bigskip

{\bf Slow solutions:} In addition to the already defined number $\gamma=\Frac{\beta-\alpha}{\alpha+1}+1 = \Frac{\beta+1}{\alpha+1}$, let us introduce $\hat{\gamma}=\Frac{\beta-l}{l+2}+1 = \Frac{\beta+2}{l+2}+1 $. The condition
\begin{eqnarray*}
\alpha <\frac{\beta(1+l)+l}{\beta+2},
\end{eqnarray*}
implies immediately
\begin{eqnarray*}\frac{\beta-l}{l+2}<\frac{\beta-\alpha}{\alpha+1}.
\end{eqnarray*}
Therefore \begin{eqnarray*}
\hat{\gamma}<\gamma.
\end{eqnarray*}

In this section, we consider the energies $H(t)$ and 
\begin{eqnarray}\label{definition_K}
K(t):=\frac{\left\Vert u^{\prime}(t)\right\Vert^{2}}{\|A^{\frac{1}{2}}u(t)\|^{2\hat{\gamma}}},
\end{eqnarray} defined whenever $u(t)\not=0$.
In this part of the proof, it remains to consider the case where (\ref{H_H3}) is false for every $\eta$. Then, assuming the solution to be non-trivial there exists a sequence $t_n \rightarrow +\infty$ such that $ u(t_n)\not=0  \quad \text{and  } \quad 
K(t_n)\longrightarrow 0$. Therefore, for any $\varepsilon_1>0$ there exists $N(\varepsilon_1)\in \mathbb{N}$ such that 
\begin{eqnarray}\label{H3_tn} u(t_n)\not=0  \quad \text{and  } \quad 
K(t_n)\leq \Frac{\varepsilon_1}{2}, \text{\ \ \ } \forall n \ge N(\varepsilon_1).
\end{eqnarray}
Under this assumption, we prove that a similar estimate holds true for all sufficiently large times,  namely for some $n_0\in \mathbb{N}$ 
\begin{eqnarray}\label{H3_t}
u(t)\not=0  \quad \text{and  } \quad K(t)\leq M, \text{\ \ \ } \forall t\ge t_{n_0},
\end{eqnarray}
for a suitable constant $M\ge \frac{\varepsilon_1}{2}$. In order to achieve this property, we select another positive number, $ \varepsilon_2 $, to be fixed later in the proof, and we consider $n_0\in \mathbb{N}$ large enough to achieve the additional condition 
$$ \|A^{\frac{1}{2}}u(t)\| \le \varepsilon_2.$$
Let us set
\begin{eqnarray*}
\bar{T}:=\sup \{t\ge t_{n_0}:\forall \tau\in [t_{n_0},t], \text{\ \ }u(\tau)\ne 0\}.
\end{eqnarray*}
By (\ref{H3_tn}), we deduce that $u(t_{n_0})\ne 0$ and we have that $\bar{T}>0$. If $\bar{T}<+\infty$, then $u(\bar{T})=0$. Let us consider the energy $K$, which is defined by (\ref{definition_K}) for every $t\in[t_{n_0},\bar{T})$. \\

If $u'(t) = 0$, then we claim that $K$ is differentiable at $t$ with $K'(t) = 0$.  Indeed by the equation we find that $\|u'\|^{l}u'(s) = O(s-t)$ for $s$ close to $t$, then  since $l<1$ we deduce $\left\Vert u^{\prime}(t)\right\Vert^{2} = O(s-t)$, thereby proving the claim. \\

If $u'(t) \not= 0$, then again $K$ is differentiable at $t$ with:
\begin{eqnarray*}
K^{\prime}(t)=\frac{\displaystyle \frac{d}{dt} \left(\left\Vert u^{\prime}(t) \right\Vert^{2+l}\right)}{\|A^{\frac{1}{2}}u(t)\|^{2\hat{\gamma}}\left\Vert u^{\prime}(t) \right\Vert^{l}}-\frac{\displaystyle \frac{d}{dt} \left(\|A^{\frac{1}{2}}u(t)\|^{2\hat{\gamma}}\left\Vert u^{\prime}(t) \right\Vert^{l}\right)}{\|A^{\frac{1}{2}}u(t)\|^{2\hat{\gamma}}\left\Vert u^{\prime}(t) \right\Vert^{l}}K(t).
\end{eqnarray*}
Taking into account (\ref{energy}) and the energy identity, we deduce 
\begin{eqnarray}\label{H3_prime}
K^{\prime}(t)&\leq&-\rho\, \eta_2\frac{\left\Vert u^{\prime}(t)\right\Vert^{\alpha+2-l}}{\|A^{\frac{1}{2}}u(t)\|^{2\hat{\gamma}}}-\rho \frac{ \|A^{\frac{1}{2}}u(t)\| ^{\beta-2\hat{\gamma}}\left ( u^{\prime}(t),Au(t)\right )}{\left\Vert u^{\prime}(t) \right\Vert^{l}} \nonumber \\
&&-\frac{\displaystyle \frac{d}{dt} \left(\|A^{\frac{1}{2}}u(t)\|^{2\hat{\gamma}}\left\Vert u^{\prime}(t) \right\Vert^{l}\right)}{\|A^{\frac{1}{2}}u(t)\|^{2\hat{\gamma}}\left\Vert u^{\prime}(t) \right\Vert^{l}}K(t)=:K^1+K^2+K^3,
\end{eqnarray}
where $\rho:=\displaystyle \frac{l+2}{l+1}$.
\\

Let us estimate $K^1$, $K^2$ and $K^3$. Using (\ref{definition_K}), we observe 
\begin{eqnarray}\label{H3_1}
K^1&=&- \rho \, \eta_2\left(\frac{\left\Vert u^{\prime}(t)\right\Vert^2}{\|A^{\frac{1}{2}}u(t)\|^{2\hat{\gamma}}} \right)^{\frac{\alpha+2-l}{2}}\frac{\|A^{\frac{1}{2}}u(t)\|^{\hat{\gamma}(\alpha+2-l)}}{\|A^{\frac{1}{2}}u(t)\|^{2\hat{\gamma}}} \nonumber\\ 
&=&-\rho \, \eta_2 K^{\frac{\alpha+2-l}{2}}(t)\|A^{\frac{1}{2}}u(t)\|^{\hat{\gamma}( \alpha-l)}.
\end{eqnarray}
In order to estimate $K^2$, we use Young's inequality applied with the conjugate exponents $\frac{\alpha+2-l}{1-l}$ and $\frac{\alpha+2-l}{\alpha+1}$,
\begin{eqnarray*}
\left\vert K^2 \right\vert &\leq& \rho\left\Vert u^{\prime}(t)\right\Vert^{1-l} \cdot \frac{\|A^{\frac{1}{2}}u(t)\|^{\beta+1}}{\|A^{\frac{1}{2}}u(t)\|^{2\hat{\gamma}}}\\
&\leq& \frac{1}{\|A^{\frac{1}{2}}u(t)\|^{2\hat{\gamma}}}\left(\delta_1 \left\Vert u^{\prime}(t)\right\Vert^{\alpha+2-l}+\delta_2  \|A^{\frac{1}{2}}u(t)\|^{\frac{(\beta+1)(\alpha+2-l)}{\alpha+1}}\right),
\end{eqnarray*}
where $\delta_1:=\displaystyle \rho\, \eta_2\left( \frac{1-l}{\alpha+2-l}\right)$ and $\delta_2:= \displaystyle \frac{\rho}{\eta_2}\left( \frac{\alpha+1}{\alpha+2-l}\right)$, and taking into account that $\gamma=\frac{\beta+1}{\alpha+1}$, we deduce
\begin{eqnarray*}
\left\vert K^2 \right\vert &\leq&\delta_1 \left(  \frac{\left\Vert u^{\prime}(t)\right\Vert^2}{\|A^{\frac{1}{2}}u(t)\|^{2\hat{\gamma}}} \right)^{\frac{\alpha+2-l}{2}}\frac{\|A^{\frac{1}{2}}u(t)\|^{\hat{\gamma}(\alpha+2-l)}}{\|A^{\frac{1}{2}}u(t)\|^{2\hat{\gamma}}}+\delta_2 \|A^{\frac{1}{2}}u(t)\| ^{(\beta+1)\left(\frac{\alpha+2-l}{\alpha+1}-\frac{2}{l+2} \right)-\frac{2}{l+2}} \nonumber\\
&=& \delta_1K^{\frac{\alpha+2-l}{2}}(t)\|A^{\frac{1}{2}}u(t)\|^{\hat{\gamma} (\alpha-l)}+\delta_2 \|A^{\frac{1}{2}}u(t)\| ^{\mu},
\end{eqnarray*}
where $\mu:={(\beta+1)\left(\frac{\alpha+2-l}{\alpha+1}-\frac{2}{l+2} \right)-\frac{2}{l+2}}  = \frac{(\beta+1)(\alpha+2-l)}{\alpha+1} - 2 {\hat{\gamma}}$.

By simple calculations, we deduce from $
\alpha <\frac{\beta(1+l)+l}{\beta+2}$
that \begin{eqnarray*}
\hat{\gamma}(\alpha-l)<\mu, 
\end{eqnarray*}
therefore
\begin{eqnarray}\label{H3_2}
\left\vert K^2 \right\vert &\leq& \|A^{\frac{1}{2}}u(t)\|^{\hat{\gamma} (\alpha-l)}( \delta_1K^{\frac{\alpha+2-l}{2}}(t)+\delta_2 \|A^{\frac{1}{2}}u(t)\| ^{\nu}),
\end{eqnarray}
with $ \nu= \mu-\hat{\gamma}(\alpha-l)>0$.

{ For simplicity in the sequel we shall write 
\begin{equation}\label{m}
l/2: = m>0.
\end{equation}
}
In order to estimate $K^3$, we use (\ref{definition_K}) and we have that
\begin{eqnarray*}
K^3&=&-\frac{\displaystyle \frac{d}{dt} \left(\|A^{\frac{1}{2}}u(t)\|^{2\hat{\gamma}(1+m)}K^m(t)\right)}{\|A^{\frac{1}{2}}u(t)\|^{2\hat{\gamma}}\left\Vert u^{\prime}(t) \right\Vert^{l}}K(t)=-mK^m(t)K^{\prime}(t)\frac{\|A^{\frac{1}{2}}u(t)\|^{2\hat{\gamma} m}}{\left\Vert u^{\prime}(t) \right\Vert^{l}}\\
&&- 2\hat{\gamma} (1+m)K^{m+1}(t)\frac{\|A^{\frac{1}{2}}u(t)\|^{2(\hat{\gamma} m-1)}}{\left\Vert u^{\prime}(t) \right\Vert^{l}}\left ( u^{\prime}(t),Au(t)\right )=:K^4+K^5.
\end{eqnarray*}
We observe that taking  account  of (\ref{definition_K}), we have
\begin{eqnarray}\label{H3_4}
K^4=-mK^{\prime}(t),
\end{eqnarray}
and
\begin{eqnarray*}
\left\vert K^5 \right\vert \leq 2 \hat{\gamma}(1+m)K^{m+1}(t)\frac{\|A^{\frac{1}{2}}u(t)\|^{2\hat{\gamma} m}}{\left\Vert u^{\prime}(t)\right\Vert^{l}}\cdot \frac{\left\Vert u^{\prime}(t)\right\Vert}{\|A^{\frac{1}{2}}u(t)\|}=2\hat{\gamma}(1+m) K^{\frac{3}{2}}(t)\|A^{\frac{1}{2}}u(t)\| ^{\hat{\gamma}-1}.
\end{eqnarray*} We observe that our condition on $\alpha$ can be written in the form 
$$ \alpha (\beta+2)< (\beta+2)(1+l)-(l+2) \Rightarrow (\beta+2) (1+l-\alpha)> l+2$$ 
$$\Rightarrow \hat{\gamma} (1+l-\alpha)> 1  \Rightarrow \hat{\gamma}-1>\hat{\gamma}(\alpha-l).$$
Using (\ref{estimate_u}) we therefore obtain
\begin{eqnarray}\label{H3_5}
\left\vert K^5 \right\vert \leq C_0K^{\frac{3}{2}}(t)\|A^{\frac{1}{2}}u(t)\| ^{\hat{\gamma} (\alpha-l)}.
\end{eqnarray} 
Then, taking into account (\ref{H3_1})-(\ref{H3_5}) in (\ref{H3_prime}), we deduce that
\begin{equation}\label{derivative_H3}
 K^{\prime}(t)\!\!\leq\! \|A^{\frac{1}{2}}u(t)\| ^{\hat{\gamma} (\alpha-l)} \left[\!-\widehat{\rho} \, \eta_2 K^{\frac{\alpha+2-l}{2}}(t)\!+C_1\|A^{\frac{1}{2}}u(t)\| ^{\nu}+ C_2K^{\frac{3}{2}}(t)\! \right],
\end{equation}
where $\widehat{\rho}:=\frac{\rho(\alpha+1)}{(1+m)(\alpha+2-l)}$.\\

Here we observe that \eqref {derivative_H3} is still valid if $u'(t) = 0$, since then $K(t) = 0 = K'(t)$ and the RHS is non-negative. Now, let $$h(s)=-\widehat{\rho} \, \eta_2 s^{\frac{\alpha+2-l}{2}}+C_1 \varepsilon_2^{\nu}+C_2 s^{\frac{3}{2}},$$  and 
$$\tilde{h}(s)=-\widehat{\rho} \, \eta_2 s^{\frac{\alpha+2-l}{2}}+C_2 s^{\frac{3}{2}}.$$  
Since $\alpha <\frac{\beta(1+l)+l}{\beta+2}$, we have $$\alpha<1+\frac{l(\beta+1)}{\beta+2}\Longrightarrow \alpha+2-l<3. $$
We observe that $$\tilde{h}(s)<0, \text{\ \ if  } s \rightarrow 0.$$ Fixing  $\varepsilon_1$ sufficiently small to achieve $\tilde{h}(\varepsilon_1)<0 $ , we can then adjust $\varepsilon_2$ so that $h(\varepsilon_1)<0. $

Now, if $K$ is not bounded for all $t\in(t_{n_0},\bar{T})$, then there exists $T_1\in(t_{n_0},\bar{T})$ such that $K(T_1)=\varepsilon_1 $ and $K(t)< \varepsilon_1 $ for all $t\in (t_{n_0},T_1)$.
 
 By (\ref{derivative_H3}) and taking into account that $u\in \mathcal{C}^1(\mathbb{R}^+,H)$ and $\left\Vert u^{\prime}\right\Vert ^{l}u^{\prime}\in \mathcal{C}^1(\mathbb{R}^+,H) $, we have $$K^{\prime}(T_1)\leq \|A^{\frac{1}{2}}u(T_1)\| ^{\hat{\gamma}( \alpha-l)}h(\varepsilon_1)\leq \left\Vert u(T_1) \right\Vert ^{\hat{\gamma}( \alpha-l)}h(\varepsilon_1)<0,$$
then since $T_1\in(t_{n_0},\bar{T})$ implies $\left\Vert u(T_1) \right\Vert>0$,  $K$ decreases near $T_1$. This contradicts the definition of $T_1$ and we can claim that $K$ is bounded for all $t\in (t_{n_0},\bar{T})$.

We claim that $\bar{T}=+\infty$. Let us assume by contradiction that this is not the case. Then, taking into account that $K$ is bounded for all $t\in [t_{n_0},\bar{T})$, we observe that
\begin{eqnarray*}
\left\Vert u^{\prime}(t)\right\Vert^{2}\leq M \|A^{\frac{1}{2}}u(t)\|^{2\hat{\gamma}}, \quad  \forall t\in[t_{n_0},\bar{T}).
\end{eqnarray*}

Therefore, from the continuity of the vector $\left(u^{\prime},u\right)$ with values in $H \times H$ it now follows that $u(\bar{T})=0$ implies $u^{\prime}(\bar{T})=0$, hence $u\equiv 0$ by backward uniqueness, a contradiction. Hence, $\bar{T}=+\infty$. Then, we obtain (\ref{H3_t}). 


So we are left to prove that for some $M>0$
\begin{eqnarray}\label{condition2}
\frac{\left\Vert u^{\prime}(t)\right\Vert^2}{\|A^{\frac{1}{2}}u(t)\|^{2\left(\frac{\beta-\alpha}{\alpha+1}+1\right)}}\leq M \text{\ \ \ } \forall t\ge t_{n_0}.
\end{eqnarray}
Let us consider the energy $H$, which is defined by (\ref{definition_H}) for every $t\ge t_{n_0}$.

Assuming first $u'(t)\not = 0$,  we can differentiate $H$ at t, and we find
\begin{eqnarray*}
H^{\prime}(t)=\frac{\displaystyle \frac{d}{dt} \left(\left\Vert u^{\prime}(t) \right\Vert^{2+l}\right)}{\|A^{\frac{1}{2}}u(t)\|^{2\gamma}\left\Vert u^{\prime}(t) \right\Vert^{l}}-\frac{\displaystyle \frac{d}{dt} \left(\|A^{\frac{1}{2}}u(t)\|^{2\gamma}\left\Vert u^{\prime}(t) \right\Vert^{l}\right)}{\|A^{\frac{1}{2}}u(t)\|^{2\gamma}\left\Vert u^{\prime}(t) \right\Vert^{l}}H(t).
\end{eqnarray*}
Taking into account (\ref{energy}) and the energy identity, we deduce 
\begin{eqnarray}\label{H1_prime}
H^{\prime}(t)&\leq&-\rho\, \eta_2\frac{\left\Vert u^{\prime}(t)\right\Vert^{\alpha+2-l}}{\|A^{\frac{1}{2}}u(t)\|^{2\gamma}}-\rho \frac{ \|A^{\frac{1}{2}}u(t)\| ^{\beta-2\gamma}\left ( u^{\prime}(t),Au(t)\right )}{\left\Vert u^{\prime}(t) \right\Vert^{l}} \nonumber \\
&&-\frac{\displaystyle \frac{d}{dt} \left(\|A^{\frac{1}{2}}u(t)\|^{2\gamma}\left\Vert u^{\prime}(t) \right\Vert^{l}\right)}{\|A^{\frac{1}{2}}u(t)\|^{2\gamma}\left\Vert u^{\prime}(t) \right\Vert^{l}}H(t)=:H^1+H^2+H^3,
\end{eqnarray}
where $\rho:=\Frac{l+2}{l+1}$.
\\

Let us estimate $H^1$, $H^2$ and $H^3$. Using (\ref{definition_H}), we observe that
\begin{eqnarray}\label{H1_1}
H^1&=&- \rho \, \eta_2\left(\frac{\left\Vert u^{\prime}(t)\right\Vert^2}{\|A^{\frac{1}{2}}u(t)\|^{2\gamma}} \right)^{\frac{\alpha+2-l}{2}}\frac{\|A^{\frac{1}{2}}u(t)\|^{\gamma(\alpha+2-l)}}{\|A^{\frac{1}{2}}u(t)\|^{2\gamma}} \nonumber\\
&=&-\rho \, \eta_2 H^{\frac{\alpha+2-l}{2}}(t)\|A^{\frac{1}{2}}u(t)\|^{\gamma( \alpha-l)}.
\end{eqnarray}
In order to estimate $H^2$, we use Young's inequality applied with the conjugate exponents $\frac{\alpha+2-l}{1-l}$ and $\frac{\alpha+2-l}{\alpha+1}$,
\begin{eqnarray*}
\left\vert H^2 \right\vert &\leq& \rho\left\Vert u^{\prime}(t)\right\Vert^{1-l} \cdot \frac{\|A^{\frac{1}{2}}u(t)\|^{\beta+1}}{\|A^{\frac{1}{2}}u(t)\|^{2\gamma}}=\rho \frac{\left\Vert u^{\prime}(t)\right\Vert^{1-l}}{\|A^{\frac{1}{2}}u(t)\|^{\frac{2\gamma(1-l)}{\alpha+2-l}}}\cdot \frac{\|A^{\frac{1}{2}}u(t)\|^{\frac{2\gamma(1-l)}{\alpha+2-l}+\beta+1}}{\|A^{\frac{1}{2}}u(t)\|^{2\gamma}}\\
&\leq& \delta_1 \left(  \frac{\left\Vert u^{\prime}(t)\right\Vert^{1-l}}{\|A^{\frac{1}{2}}u(t)\|^{\frac{2\gamma(1-l)}{\alpha+2-l}}} \right)^\frac{\alpha+2-l}{1-l}+\delta_2 \left( \|A^{\frac{1}{2}}u(t)\|^{\frac{2\gamma(1-l)}{\alpha+2-l}+\beta+1-2\gamma}\right)^{\frac{\alpha+2-l}{\alpha+1}},
\end{eqnarray*}
where $\delta_1:=\displaystyle \rho\, \eta_2\left( \frac{1-l}{\alpha+2-l}\right)$ and $\delta_2:= \displaystyle \frac{\rho}{\eta_2}\left( \frac{\alpha+1}{\alpha+2-l}\right)$, and taking into account that $\gamma=\frac{\beta+1}{\alpha+1}$, we deduce
\begin{eqnarray}\label{H1_2}
\left\vert H^2 \right\vert &\leq&\delta_1 \left(  \frac{\left\Vert u^{\prime}(t)\right\Vert^2}{\|A^{\frac{1}{2}}u(t)\|^{2\gamma}} \right)^{\frac{\alpha+2-l}{2}}\frac{\|A^{\frac{1}{2}}u(t)\|^{\gamma(\alpha+2-l)}}{\|A^{\frac{1}{2}}u(t)\|^{2\gamma}}+\delta_2 \|A^{\frac{1}{2}}u(t)\| ^{\gamma( \alpha-l)} \nonumber\\
&=& \delta_1H^{\frac{\alpha+2-l}{2}}(t)\|A^{\frac{1}{2}}u(t)\|^{\gamma (\alpha-l)}+\delta_2 \|A^{\frac{1}{2}}u(t)\| ^{\gamma( \alpha-l)}.
\end{eqnarray}
In order to estimate $H^3$, using (\ref{m}) and (\ref{definition_H}), we have that
\begin{eqnarray*}
H^3&=&-\frac{\displaystyle \frac{d}{dt} \left(\|A^{\frac{1}{2}}u(t)\|^{2\gamma(1+m)}H^m(t)\right)}{\|A^{\frac{1}{2}}u(t)\|^{2\gamma}\left\Vert u^{\prime}(t) \right\Vert^{l}}H(t)=-mH^m(t)H^{\prime}(t)\frac{\|A^{\frac{1}{2}}u(t)\|^{2\gamma m}}{\left\Vert u^{\prime}(t) \right\Vert^{l}}\\
&&- 2\gamma (1+m)H^{m+1}(t)\frac{\|A^{\frac{1}{2}}u(t)\|^{2(\gamma m-1)}}{\left\Vert u^{\prime}(t) \right\Vert^{l}}\left ( u^{\prime}(t),Au(t)\right )=:H^4+H^5.
\end{eqnarray*}
We observe that taking into account (\ref{definition_H}), we have
\begin{eqnarray}\label{H1_4}
H^4=-mH^{\prime}(t),
\end{eqnarray}
and using (\ref{definition_K}), we obtain
\begin{eqnarray}\label{H1_5}
\left\vert H^5 \right\vert &\leq& 2 \gamma(1+m)H^{m+1}(t)\frac{\|A^{\frac{1}{2}}u(t)\|^{2\gamma m}}{\left\Vert u^{\prime}(t)\right\Vert^{l}}\cdot \frac{\left\Vert u^{\prime}(t)\right\Vert}{\|A^{\frac{1}{2}}u(t)\|}\\ \nonumber
&=&2\gamma(1+m)H(t) K^{\frac{1}{2}}(t)\|A^{\frac{1}{2}}u(t)\| ^{\hat{\gamma}-1}.
\end{eqnarray} 
We observe, after some calculations, that 
\begin{eqnarray*}
\hat{\gamma}-1-\gamma(\alpha-l)>0 \Longleftrightarrow \alpha(\beta+2)(1+l)<l(\beta+1)(l+2)+\beta-l,
\end{eqnarray*}
which reduces easily to the condition
\begin{eqnarray*}
\alpha <\frac{\beta(1+l)+l}{\beta+2}.
\end{eqnarray*}
Then, taking into account (\ref{H1_1})-(\ref{H1_5}) in (\ref{H1_prime}), we deduce that
\begin{eqnarray*}
\!\!\!\!\!H^{\prime}(t)&\!\!\leq\!\!& \|A^{\frac{1}{2}}u(t)\| ^{\gamma (\alpha-l)} \left[-\widehat{\rho} \, \eta_2 H^{\frac{\alpha+2-l}{2}}(t)+\frac{\widehat{\rho}}{\eta_2}+2\gamma H(t)K^{\frac{1}{2}}(t)\|A^{\frac{1}{2}}u(t)\|^{\hat{\gamma}-1-\gamma(\alpha-l)} \right]\!,
\end{eqnarray*}
where $\widehat{\rho}:=\frac{\rho(\alpha+1)}{(1+m)(\alpha+2-l)}$. Now, we estimate the third term. We have that $K$ is bounded for all $t\ge t_{n_0}$. Then, using Young's inequality applied with the conjugate exponents $\frac{\alpha+2-l}{2}$ and $\frac{\alpha+2-l}{\alpha-l}$,
\begin{eqnarray*}
2\gamma H(t)K^{\frac{1}{2}}(t)\|A^{\frac{1}{2}}u(t)\|^{\hat{\gamma}-1-\gamma(\alpha-l)} &\leq& \widetilde{M}H(t)\\
&\leq& \frac{\widehat{\rho}\,\eta_2}{2}H^{\frac{\alpha+2-l}{2}}(t)+\frac{4(\alpha-l)}{\widehat{\rho}\,\eta_2(\alpha+2-l)^2}\widetilde{M}^{\frac{\alpha+2-l}{\alpha-l}},
\end{eqnarray*}
where $\widetilde{M}:=2\gamma M^{\frac{1}{2}}\varepsilon_2^{\hat{\gamma}-1-\gamma(\alpha-l)}$.

Then, we obtain
\begin{eqnarray*}
\!\!\!\!\!H^{\prime}(t)&\!\!\leq\!\!& \|A^{\frac{1}{2}}u(t)\| ^{\gamma (\alpha-l)} \left[-\frac{1}{2}\widehat{\rho} \, \eta_2 H^{\frac{\alpha+2-l}{2}}(t)+\widehat{K}\right]\!,
\end{eqnarray*}
where $\widehat{K}:=\displaystyle\frac{\widehat{\rho}}{\eta_2}+\frac{4(\alpha-l)}{\widehat{\rho}\,\eta_2(\alpha+2-l)^2}\widetilde{M}^{\frac{\alpha+2-l}{\alpha-l}}$. 
If $u'(t) = 0$, since as previously $H$ is differentiable at $t$ with $H'(t) = 0 = H(t)$, the inequality is still valid. 

It remains to prove that $H(t)$ is bounded for all $t\ge t_{n_0}$. Indeed if $H$ is not bounded for all $t \ge t_{n_0}$, then there exist $T_2 \ge t_{n_0}$ such that 
$$
T_2:=\inf \{t\ge t_{n_0}: \,\,\forall \tau \in [t_{n_0},t], \,\, \frac{1}{2}\widehat{\rho} \, \eta_2  H^{\frac{\alpha+2-l}{2}}(\tau) <\frac{1}{2}\widehat{\rho} \, \eta_2  H^{\frac{\alpha+2-l}{2}}(t_{n_0})  +\widehat{K} +1\}.
$$
We have 
\begin{eqnarray*}
H'(T_2) &\leq& \|A^{\frac{1}{2}}u(T_2)\| ^{\gamma (\alpha-l)} \left[-\frac{1}{2}\widehat{\rho} \, \eta_2 H^{\frac{\alpha+2-l}{2}}(T_2)+\widehat{K}\right]\\
&= & -\|A^{\frac{1}{2}}u(T_2)\| ^ {\gamma (\alpha-l)}  \left[1+ \frac{1}{2}\widehat{\rho} \, \eta_2  H^{\frac{\alpha+2-l}{2}}(t_{n_0})\right]<0,
\end{eqnarray*}
and since $T_2\ge t_{n_0},$ implies $u(T_2)\not = 0$, $H$ decreases near $T_2$. This contradicts the definition of $T_2$ and we can claim that $H$ is bounded for all $t\ge t_{n_0}$. The end of the proof is identical to that of Theorem \ref{Main_Theorem}. 

\end{proof}

\begin{remark}\rm{In the special case $A = \lambda I $, by using the scalar equation we can prove the existence of  an $N$ - dimensional variety of fast solutions. Actually, by Theorem 5.4, (4) of \cite {GhGoHa} we know that the same result is true in general for the non-singular linearly damped equation $$ u''+ u' +  \|A^{1/2}u\|^\beta Au = 0. $$  The question now arises  of whether the dimensionality of the variety of fast solutions is exactly equal to $N$ or possibly larger. The special case of the equation \begin{equation}\label{rad} u''+ u' +  \|u\|^\beta u = 0 \end{equation} shows it is exactly equal to $N$ at least in some cases. Indeed we claim that in this case, all fast solutions have their $N$ components proportional to the same fast solution of a related scalar equation.  \begin{proof}Given a solution $u$ of \eqref{rad}, any component $y$ of $u$ satisfies  $$ y''+ y' +  \|u\|^\beta y = 0. $$ Considering two different components $(y, z)$ we obtain that the Wronskian function $w(t) :=(y'z-yz')(t)$ is
 a solution of 
$$ w'+ w = 0 .$$ Hence $w(t) = w(0) \exp (-t)$. Since on the other hand $w$, being a quadratic combination of 4 functions decaying like  $\exp (-t)$, has to decay like $\exp (-2t)$, the only possibility is $w(0) = 0$, hence $w=0$. Hence any two components are proportional. Let us set $$ y_j = a_j y, $$ where $y$ is one of the non-zero components. Then we have $$ y''+ y' +  m\|y\|^\beta y = 0$$ with $m := {\{\sum_{j=1}^N a_j^2\}}^{\beta/2}$. Let us introduce $ \phi : = \mu y $ where $\mu>0$ will be chosen later. Then $\phi$ satisfies 
$ \phi''+ \phi' +  \mu m\|y\|^\beta y = 0$ or $$ \phi''+ \phi' +  \mu^{-\beta} m\|\phi\|^\beta \phi= 0, $$ so that the choice $\mu = m^{1/\beta}$ leads to 
$$ \phi''+ \phi' + \|\phi\|^\beta \phi= 0. $$  Finally we have $$ y_j = b_j \phi, $$ with $b_j = m^{-1/\beta} a_j$ and $\phi$ a fast solution of the scalar equation. Since the set of such fast solutions $\phi$ is invariant by translation in $t$, for the moment we reduced the dimension to $N+1$ instead of $N$ and we know that this is not relevant since in the scalar case the dimension is $1$. Actually the numbers $b_j$ are constrained by an additional condition: indeed we have for all $j$ the equation $$ y_j'' + y_j' +\|u\|^\beta y_j = b_j(  \phi''+ \phi' + \|u\|^\beta \phi) = 0, $$ and since at least one of the $b_j$ is not $0$, we obtain $ \phi''+ \phi' + \|u\|^\beta \phi = 0 .$ Now $$ \|u\|^\beta = {\Big\{\sum_{j=1}^N b_j^{2} \Big\}}^{\beta/2}\|\phi\|^\beta. $$ Since slow solutions do not vanish for $t$ large we conclude that $ \sum_{j=1}^N b_j^{2}  = 1$, the additional constraint we were looking for, reducing the dimension to $N$. The set of slow solutions has the structure of a 1D curve fibrated by the $(N-1)$-dimensional unit sphere. \end{proof}
However, when either $g$ is non-linear or $A$ is not a multiple of the identity, we have no information on the dimension of this set.} This is one of the last open problems of this theory. \end{remark} 

\subsection*{Acknowledgments}
M. Anguiano has been supported by Junta de Andaluc\'ia (Spain), Proyecto de Excelencia P12-FQM-2466, and in part by  European Commission, Excellent Science-European Research Council (ERC) H2020-EU.1.1.-639227.


\begin{thebibliography}{10}
\bibitem{AMH} M. Abdelli and A. Haraux, {\em Global behavior of the solutions
to a class of nonlinear second order ODE's\/,} Nonlinear Analysis, {\bf 96} (2014), 18-73.
\bibitem{AIH} F. Aloui, I. Ben Hassen and A. Haraux  {\em Compactness of trajectories to some nonlinear second order evolution equations and applications\/,} Journal de Mathematiques pures et appliquees , {\bf 100}  No, 3 (2013), 295-326.
\bibitem{GGH} M. Ghisi, M. Gobbino and A. Haraux, {\em Optimal decay estimates for the general solution to a class of 
semi-linear dissipative hyperbolic equations,} J. Eur. Math. Soc. (JEMS) {\bf 18} (2016), no. 9, 1961--1982.
 \bibitem{GhGoHa} M. Ghisi, M. Gobbino and A. Haraux, {\em Finding the exact decay rate of all solutions to some second order evolution equation with dissipation\/,} J. Funct. Anal. {\bf 271}  (2016), no. 9, 2359--2395.
\bibitem{har2} A. Haraux, {\em Sharp decay estimates of the solutions to a class
of nonlinear second order ODE's\/,} Analysis and Applications,
{\bf 9} (2011), 49-69.
\bibitem{har3} A. Haraux, {\em Slow and fast decay of solutions to some second order evolution equations\/,} J. Anal. Math. {\bf 95} (2005), 297-321.
\end{thebibliography}
\end{document}